\documentclass[final, onefignum]{siamart190516}

\usepackage[utf8]{inputenc}
\usepackage{bm}
\usepackage{amsmath, amsfonts,amssymb}
\usepackage{mathtools}
\usepackage{tgpagella}
\newsiamremark{remark}{Remark}

\let\epsilon\varepsilon

\let\phi\varphi

\let\theta\vartheta

\newcommand{\tC}{{\widetilde{C}}}

\newcommand{\hv}{\hat{v}}
\newcommand{\hrho}{{\widetilde{\rho}}}

\newcommand{\ts}{{\tilde{s}}}
\newcommand{\uLambda}{{\underline{\Lambda}}}
\newcommand{\utgamma}{{\underline{\widetilde{\gamma}}}}

\newcommand{\dalpha}{\partial^\alpha}

\newcommand{\ddelta}{\partial^\delta}
\newcommand{\dxi}{\partial^\xi}

\newcommand{\dK}{{\partial K}}

\newcommand{\dbeta}{\partial^\beta}
\newcommand{\hdbeta}{\hat{\partial}^\beta}
\newcommand{\hdalpha}{\hat{\partial}^\alpha}

\newcommand{\betam}{{|\beta|}}

\newcommand{\xim}{{|\xi|}}
\newcommand{\alpham}{{|\alpha|}}
\newcommand{\deltam}{{|\delta|}}

\newcommand{\hr}{\hat{r}}

\newcommand{\tA}{{\widetilde{A}}}

\newcommand{\ugamma}{{\underline{\gamma}}}
\newcommand{\ubeta}{{\underline{\beta}}}
\newcommand{\ueta}{{\underline{\eta}}}

\newcommand{\cK}{\mathcal{K}}

\newcommand{\cP}{\mathcal{P}}
\newcommand{\cS}{\mathcal{S}}
\newcommand{\bcK}{\bm{\mathcal{K}}}
\newcommand{\bcU}{\bm{\mathcal{U}}}
\newcommand{\bu}{\bm{u}}
\newcommand{\bU}{\bm{U}}
\newcommand{\bw}{\bm{w}}
\newcommand{\bbf}{\bm{f}}
\newcommand{\bzero}{{\bm{0}}}

\newcommand{\chihat}{\widehat{\chi}}
\renewcommand{\Re}{\operatorname{Re}}

\headers{Analytic regularity for Navier-Stokes in polygons}{Marcati and Schwab}

\ifpdf
\hypersetup{
  pdftitle={Analytic regularity for Navier-Stokes equations in polygons},
  pdfauthor={Marcati and Schwab}
}
\fi

\author{Carlo Marcati\thanks{Seminar of Applied Mathematics, ETH Zurich, Zürich,
    Switzerland
  (\email{carlo.marcati@sam.math.ethz.ch}, \email{christoph.schwab@sam.math.ethz.ch}).}
\and Christoph Schwab\footnotemark[1]}
\title{Analytic regularity for the incompressible Navier-Stokes equations in polygons}
\begin{document}
\maketitle
\begin{abstract}
In a plane polygon $P$ with straight sides,
we prove analytic regularity of the Leray-Hopf solution of the 
stationary, viscous, and incompressible Navier-Stokes equations.
We assume small data, analytic volume force and no-slip boundary conditions.
Analytic regularity is quantified in so-called countably normed,
corner-weighted spaces with homogeneous norms. 
Implications of this analytic regularity
include exponential smallness of Kolmogorov $N$-widths of solutions,
exponential convergence rates of mixed $hp$-discontinuous Galerkin finite element
and spectral element discretizations and of model order reduction techniques.
\end{abstract}
\begin{keywords}
Navier-Stokes equations, analytic regularity, conical singularities, weighted Sobolev spaces.
\end{keywords}
\begin{AMS}
  35Q30, 76N10, 35A20
\end{AMS}
\section*{Introduction}
The mathematical theory of regularity of solutions to 
the viscous, incompressible 
Navier-Stokes equations has attracted enormous attention during the
past decades. 
Besides elucidating subtle scientific issues on their validity
in particular physical settings, and properties of their
solutions for small viscosity, mathematical regularity results
on their solutions have immediate consequences in a wide range 
of applications, from aerodynamics to physiological flows as, e.g.,
in hemodynamics.
 
The precise mathematical characterization of solution regularity 
in function spaces of Sobolev or Besov type is moreover essential
in the numerical analysis of approximation schemes for the 
efficient numerical solution of these equations. 
We mention only finite element, finite volume, 
$hp$- and spectral element methods and the more recent
model order reduction (MOR) and 
reduced basis (RB) methods \cite{Hesthaven2016,Quarteroni2016}.

In a plane polygonal domain $P$,
we consider the stationary Stokes and Navier-Stokes
equations modeling viscous, incompressible flow, 
i.e. 
\begin{equation}
  \label{eq:NSEIntr}
\begin{split}
  -\nu \Delta \bu +\left( \bu\cdot\nabla \right)\bu+ \nabla p  &= \bbf \text{ in }P,
  \\
  \nabla \cdot \bu &= 0  \text{ in }P,
  \\
  \bu &= \bzero \text{ on }\partial P.
\end{split}
\end{equation}
Here, the 
\emph{source term $\bbf$ is assumed analytic in weighted Sobolev spaces in $P$},
and, together with the viscosity $\nu$, it is supposed to satisfy a 
smallness assumption in $L^2(P)$ which ensures uniqueness of weak (Leray-Hopf) solutions.
We establish analytic estimates on the Leray-Hopf solution of
\eqref{eq:NSEIntr} in weighted Sobolev spaces in $P$.

For a classical analysis of the Stokes and Navier-Stokes problems,
see \cite{Ladyzhenskaya1969, Girault1986a, Temam1995}; 
the analyticity of the solutions in smooth domains
-- given sufficiently regular right hand sides -- is a classical result
\cite{Masuda1967a, Giga1983}. 
In our setting, though, special care
has to be taken to account for the corners of the domain.
Due to the corner singularities of $P$, 
solutions of the Navier-Stokes equations do not admit high order regularity in scales of classical Sobolev spaces.
Regularity results on solutions of elliptic PDEs in polygonal domains $P$ 
must account for the presence of corner singularities at the vertices of $P$.
Accordingly, 
scales of \emph{corner-weighted Sobolev spaces} are required
which were introduced in the seminal paper \cite{Kondratev1967} and are 
known today as Kondrat'ev spaces.
They extend classical Sobolev spaces by weighting weak derivatives 
with powers of the distance to (sets of) isolated singular points.
In the presently considered setting of a polygon $P$ 
and of constant kinematic viscosity $\nu$,
the set of singular points comprises the corners of $P$. 
Since \cite{Kondratev1967}, 
Kondrat'ev spaces (with their extension to three dimensional polyhedra $P$
where one considers also the distance from edges) have been widely used
for the regularity analysis of elliptic problems in polygons and polyhedra. 
Among the vast literature on the regularity of solutions of elliptic boundary value
 problems, we mention \cite{Grisvard1985, Kozlov1997,
  Nicaise1997, Kozlov2001, Costabel2005, Mazya2010, Costabel2010b,
  Costabel2010a, Costabel2012a, Costabel2014} for general elliptic problems 
  and in particular \cite{Guo2006a} for analytic regularity results for 
  the Stokes equation in a polygon as also considered in the present paper.
The regularity of solutions of the incompressible 
Stokes and Navier-Stokes equations is also
investigated from the point of view of Kondrat'ev spaces of finite order 
in the monographs \cite{Grisvard1985, Kozlov2001, Mazya2010}.

In the present paper, 
we are in particular interested in \emph{weighted analytic regularity}.
I.e., in a-priori estimates of norms of the solution in 
function spaces of Kondrat'ev type of arbitrary high order in $P$, 
with quantitative control of the growth of constants 
with respect to the order of differentiation. 
We refer to 
\eqref{eq:Kanalytic} below for the definition of the space of functions 
with weighted analytic regularity in a plane sector. 
The notion of \emph{weighted analytic regularity}, 
as developed in \cite{BDC85,Guo1986a,Guo1986b}
fits well with the classical concept of analytic regularity 
for elliptic systems in the interior and up to analytic boundary 
as developed in \cite{MorrNiren57}.
It is also of great importance for the numerical analysis of these problems, since it is the
basis for the derivation of, among others, exponentially convergent, so-called $hp$-finite
element and spectral element algorithms, as in \cite{Guo1986a, Guo1986b, Schotzau2013a, Schotzau2013b}, 
and tensor compression methods, see \cite{Kazeev2018}.

Among the articles already cited, we refer the reader to \cite{Costabel2012a, Costabel2014} 
for a reference on weighted analytic regularity in polyhedral domains 
for linear elliptic boundary value problems. 
The steady-state, incompressible viscous Navier-Stokes equations constitute 
a nonlinear elliptic system with analytic (quadratic) nonlinearity seems to preclude 
generalizing the proof of analytic regularity shifts in scales of corner-weighted, 
Hilbertian Sobolev spaces.
To handle the quadratic nonlinearity, in the present paper 
we propose to work, instead, in a non Hilbertian setting.
Specifically, we consider $L^s$-type norms, for $s > 2$, 
as is done in \cite{DallAcqua2012} for non weighted spaces and for a
different problem, and bootstrap weighted regularity by induction with respect to the
differentiation order, as opposed to arguments by local analytic continuation 
and complex analysis as put forward e.g.~in \cite{Masuda1967a,Babuska1988,MorrNiren57}.
This requires the introduction of
some preliminary regularity results, localised around each corner.
We start therefore, in
Section \ref{sec:sector}, by considering the case of an unbounded plane sector.
First, we obtain a regularity shift result in concentric
balls around the corner for the Stokes equation; then, when considering the
Navier-Stokes equation, we can move the nonlinear term to the right hand side
and show a regularity shift estimate for the solutions to the full nonlinear
equation. 
The choice of $L^s$-type spaces with $s > 2$ 
deviates from earlier, ``Hilbertian'' arguments as developed e.g. 
    in  \cite{Babuska1988,Guo2006a}. 
    It is
essential in our treatment of the quadratic nonlinearity in
the proof of the analytic regularity shift -- 
the key result Lemma \ref{lemma:induction-step} can not be established
    with our strategy when $s\leq 2$, see also Remark \ref{remark:s>2}.
This weighted a-priori estimate in Lemma \ref{lemma:induction-step} 
serves as local induction step in the proof of analytic regularity 
of (Leray-Hopf solutions of) the 
Dirichlet problems for the viscous, incompressible
Navier-Stokes equations in polygons.
The proof of analytic regularity of the corresponding velocity and
pressure fields is given in Section \ref{sec:polygon}.
There, the shift estimates for the localized problems at 
each corner of the polygon are combined with classical interior 
regularity results to obtain analytic estimates in the whole polygonal domain.
\subsection*{Main result}
Theorem \ref{th:analytic-P} is the main result of this paper: we state it
here for the convenience of the reader. We assume that $P$ is a polygonal domain
with straight sides and with $n\in \mathbb{N}$ vertices $c_i$ and internal angles $\phi_i$, $i=1, \dots, n$.
We write $r_i = |x-c_i|$, $\ugamma = \{\gamma_i\}_{i=1,\dots, n}$ and
\begin{equation*}
  r^\ugamma = \prod_{i=1, \dots, n} r_i^{\gamma_i}.
\end{equation*}
We denote by $\mathfrak{A}^i(\lambda)$ the Mellin symbol of the Stokes operator associated to
the corner $c_i$, i.e., the operator $\mathfrak{L}(\theta, \partial_\theta;
\lambda)$ in $(0, \phi_i)$, see \eqref{eq:L-symbol} below, with boundary operator 
$\mathfrak{B}(\theta, \partial_\theta; \lambda)_{|_{\theta= 0, \phi_i}}$ 
as introduced below in Section \ref{sec:stokes}. 
Let $\Lambda^i_{\min{}}$ be the real part of the eigenvalue with smallest
positive real part of $\mathfrak{A}^i(\lambda)$, and 
define $\uLambda~=~\{ \Lambda^i_{\min{}}\}_{i=1}^{n} >0$.
We remark that for each $i=1, \dots, n$, the value of $\Lambda^i_{\min{}}$ only depends on the
corresponding internal angle $\phi_i$ of the polygon $P$.

The main result of this paper is then given by the following statement;  note that
the inequality on $\ugamma$ in the hypothesis holds componentwise.
\begin{theorem}
  \label{th:analytic-P-intro}
  Let $2 < s<\infty$, $\ugamma \in \mathbb{R}^n$ such that
  $\max(-\uLambda , -1)< \ugamma-2/s < \min(\uLambda, 1)$. Let
  $(\bu, p)$ be Leray-Hopf solutions to the stationary, viscous, and incompressible 
Navier-Stokes equation in $P$ \eqref{eq:NSEIntr} 
with source term $\bbf\in \bm{H}^{-1}(P)$ such that the solution $(\bu, p)$ is unique and
that there
exist $C_f, A_f>0$ such that
  \begin{equation*}
   \| r^{\alpham + 2 - \ugamma} \dalpha \bbf \|_{L^s(P)}  \leq C_f A_f^{\alpham}\alpham! \qquad \forall\, \alpha\in \mathbb{N}^2_0.
  \end{equation*}
  Then, 
there exist two positive constants $C$ and $A$ 
such that for all $ \alpha = (\alpha_1,\alpha_2)\in \mathbb{N}_0^2:$ 
\begin{equation*}
  \| r^{\alpham - \ugamma} \dalpha \bu \|_{L^s(P)} \leq C A^\alpham \alpham !, \qquad
  \| r^{\alpham - \ugamma} \dalpha p \|_{L^s(P)} \leq C A^\alpham \alpham ! .
\end{equation*}
\end{theorem}

\subsection*{Extensions and consequences}
The analysis presented here is done in the framework of homogeneous Dirichlet boundary
conditions, and it holds without modification as long as 
homogeneous Dirichlet conditions for the velocity field
are imposed on at least one of the edges abutting on every corner.
If one wants to deal with boundary conditions where the value of the 
velocity field is not prescribed at the corner -- as, for example, for 
Neumann-type boundary conditions -- the
homogeneous weighted spaces have to be replaced by their non homogeneous
version, as in the analytic regularity shift for the Stokes system, for example, in \cite{Guo2006}. 
Non homogeneous spaces have a different weight structure than homogeneous ones, and they
contain functions with non trivial Taylor expansion at the corners. For an overview of the
relationship between homogeneous and non homogeneous spaces for isolated
point singularities, see \cite{Kozlov1997, Costabel2010a}.
The analysis presented here can also constitute part of the basis for the proof of
weighted analytic regularity of the solutions to the Navier-Stokes equation in
three dimensional polyhedra. The spaces involved in this case are more
complicated due to the presence of edge singularities, but most of the
regularity analysis in weighted spaces has been done in \cite{Mazya2010}. See
also \cite{Ebmeyer2001} for an analysis of regularity in classical Sobolev spaces.
Furthermore, the analysis in Kondrat'ev spaces can be
instrumental in other types of regularity analyses, for example using weighted
spaces to determine the regularity in Besov spaces of the solutions to the
Navier-Stokes equation, see, e.g., \cite{Eckhardt2018}. 
The analytic regularity shown here allows to establish 
exponential convergence of several classes of numerical schemes, 
such as the above-mentioned $hp$-finite element methods \cite{Schotzau2020},
tensor compression methods \cite{CSMSuri99,NSVEM18,CckbGKanDoMiScho05}
and Model Order Reduction methods \cite{Veroy2005, Lassila2014, Fick2018}
for the numerical approximation of the viscous, incompressible Navier-Stokes equations.
\section{The Stokes equation in an infinite sector}
\label{sec:sector}
In this section we consider the Stokes and Navier-Stokes equations 
in an unbounded, plane sector $K$.
Starting from an elliptic regularity shift theorem in the plane sector, we
derive, in Proposition \ref{th:stokes-estimate}, a localised regularity shift
theorem in concentric balls around the corner. 
The constants in the estimate are explicit with respect to the distance between the balls, 
as the careful tracking of the constants will be important for the analysis of the
Navier-Stokes equation.
\subsection{Notation and setting}
\label{sec:NotatSett}
Let $r$ and $\theta$ be the polar coordinates of a point $x\in \mathbb{R}^2$,
and, for an aperture angle $\phi\in (0, 2\pi)\setminus\{\pi\}$, we 
consider the plane sector
\begin{equation} \label{eq:cone}
K = \left\{ x \in \mathbb{R}^2: 
x = r(\sin(\theta), \cos(\theta))^\top, 0< r<\infty, \theta \in (0, \phi) \right\}.
\end{equation}
We exclude the case of a domain with a slit, i.e., $\phi = 2\pi$,
in order to render $K$ Lipschitz in a neighborhood of the origin.
For integer $\ell\geq 0$, $\gamma\in \mathbb{R}$, and for $1<s<\infty$, 
the homogeneous, weighted Sobolev spaces $\cK^{\ell,s}_\gamma(K)$ are
\begin{equation*}
 \cK^{\ell,s}_\gamma(K) = \left\{v\in L^s_{\mathrm{loc}}(K): r^{\alpham-\gamma}\dalpha u\in L^s(Q), \;\forall\alpha,\alpham\leq \ell\right\}
  \end{equation*}
  with norm

\begin{equation*}
\| w \|_{\cK^{\ell,s}_\gamma(K)} 
= 
\left( \sum_{\alpham\leq \ell} \|r^{\alpham -\gamma}\dalpha w\|^s_{L^s(K)} \right)^{1/s}.
\end{equation*}
We will also use the seminorm
  
\begin{equation*}
| w |_{\cK^{\ell,s}_\gamma(K)} 
= 
 \sum_{\alpham= \ell} \|r^{\alpham -\gamma}\dalpha w\|_{L^s(K)} .
\end{equation*}
For $\ell \geq 1$,
the spaces $\cK^{\ell-1/s,s}_{\gamma-1/s}(\dK\setminus\{0\})$ are defined as
spaces of traces of functions in $\cK^{\ell,s}_\gamma(K)$.
We denote 
\begin{equation*}
\cK^{\infty,s}_\gamma(K) = \bigcap_{\ell\in\mathbb{N}} \cK^{\ell,s}_\gamma(K)
\end{equation*}
and use boldface characters 
for two dimensional vectors of functions and for their spaces.
For constants $A,C>0$, we write
\begin{equation}
  \label{eq:Kanalytic}
  \cK^{\varpi,s}_\gamma(K; A,C) 
   = 
\left\{ v\in \cK^{\infty,s}_\gamma(K): 
       \left\| v\right\|_{\cK^{\ell,s}_\gamma(K)}\leq CA^\ell\ell!, \, 
                         \text{for all }\ell\in\mathbb{N} \right\}.
\end{equation}
With these sets we associate
\begin{equation}
  \label{eq:Kanalytic'}
  \cK^{\varpi,s}_\gamma(K)
   = \bigcup_{A,C>0} \cK^{\varpi,s}_\gamma(K; A,C) \;,
\end{equation}
i.e., 
\begin{equation*}
  \cK^{\varpi,s}_\gamma(K) =
\left\{ v\in \cK^{\infty,s}_\gamma(K): {\rm ex.}\, A,C>0\; \mbox{s.t.}\;\forall \ell\in \mathbb{N}:\;\;
       \left\| v\right\|_{\cK^{\ell,s}_\gamma(K)}\leq CA^\ell\ell!
                         \right\}.
\end{equation*}
Furthermore, 
given two functions or operators $F$ and $G$, 
$\left[ \cdot,\cdot \right]$ denotes the commutator
\begin{equation*} \left[ F, G \right] = FG-GF.  \end{equation*}
We write $\mathbb{N}_0 = \{0\}\cup \mathbb{N}$, 
where $\mathbb{N}$ is the set of the natural numbers.
For two multi indices 
$\alpha = (\alpha_1, \alpha_2)\in \mathbb{N}_0^2$ and 
$\beta = (\beta_1, \beta_2)\in\mathbb{N}_0^2$, 
we write 
$\alpha! =\alpha_1! \alpha_2!$, $\alpha + \beta = (\alpha_1+\beta_1,  \alpha_2+ \beta_2)$, 
and
\begin{equation*}
  \binom{\alpha}{\beta} = \frac{\alpha!}{\beta! (\alpha-\beta)!}.
\end{equation*}
Finally, recall e.g. from \cite{Kato1996} 
\begin{equation*}
  \sum_{\substack{\betam = n\\\beta\leq\alpha}} \binom{\alpha}{\beta} = \binom{\alpham}{n}.
\end{equation*}

\subsection{The Stokes equation in a sector}
\label{sec:stokes}
In the infinite sector $K$, 
we consider the Stokes system subject to volume force $\bbf$, with unit viscosity and 
with homogeneous Dirichlet boundary conditions
\begin{equation}
  \label{eq:stokes}
  \begin{aligned}
    -\Delta \bu + \nabla p  &= \bbf \text{ in }K,\\
    \nabla \cdot \bu &= g  \text{ in }K,\\
    \bu &= \bzero \text{ on }\partial K\setminus\{0\},
    \end{aligned}
\end{equation}
which we express symbolically as 
\begin{equation}\label{eq:stokesOp}
\mathcal{L} \left(\begin{array}{c} \bu \\ p \end{array}\right) 
= 
\left(\begin{array}{c} \bbf \\ g \end{array}\right) \;.
\end{equation}
With $(r,\vartheta)\in (0,\infty)\times(0,\varphi)$ we denote polar coordinates
in the sector $K$. 
For $\lambda\in \mathbb{C}$, 
we write
\begin{equation}\label{eq:Uptheta}
\bu = r^\lambda \bU(\vartheta) = r^\lambda\begin{pmatrix} U_r(\theta) \\ U_\theta(\theta)\end{pmatrix}
\quad\text{and}\quad
p=r^{\lambda-1} P(\vartheta),
\end{equation}
where $U_r$ and $U_\theta$ are respectively the radial and angular components of $\bU$.
We introduce the parameter-dependent differential operator 
$\mathfrak{L}(\theta, \partial_\theta; \lambda)$ 
acting on $(\bU,P)$ in \eqref{eq:Uptheta} via
\begin{equation*}
\mathfrak{L}(\theta, \partial_\theta; \lambda)
\begin{pmatrix}
  \bU \\ P 
\end{pmatrix}
=
\begin{pmatrix}
  r^{2-\lambda}&& \\
  & r^{2-\lambda}& \\
  && r^{1-\lambda}
\end{pmatrix}
\mathcal{L} 
\begin{pmatrix}
 r^\lambda \bU(\vartheta) 
\\ 
r^{\lambda-1} P(\vartheta) 
\end{pmatrix}
\end{equation*}
We find (see also \cite[Section 5.1, Eqn. (5.1.4)]{Kozlov2001})
\begin{equation}
  \label{eq:L-symbol}
  \mathfrak{L}(\theta, \partial_\theta; \lambda) = 
    \begin{pmatrix}
    1 -\lambda^2 - \partial_\theta^2 & 2\partial_\theta &  \lambda -1     \\
    -2\partial_\theta & 1-\lambda^2 - \partial_\theta^2 & \partial_\theta \\ 
    \lambda +1 & \partial_\theta & 0
  \end{pmatrix},
\end{equation}
and the corresponding Dirichlet boundary operator
\begin{equation*}
  \mathfrak{B}(\theta, \partial_\theta; \lambda)_{|_{\theta = 0, \phi}} =
  \begin{pmatrix}
    1 & 0 & 0 \\
    0 & 1 & 0
  \end{pmatrix}.
\end{equation*}
We denote by $\mathfrak{A}({\lambda})$ the parametric operator pencil
associated to the problem given by the parametric differential 
operator $\mathfrak{L}(\theta, \partial_\theta; \lambda)$ 
in $\theta \in (0, \phi)$ with Dirichlet
boundary operator $\mathfrak{B}(\theta, \partial_\theta; \lambda)_{|_{\theta = 0, \phi}}$.
We recall from \cite[Theorems 1.2.5 to 1.2.8]{Mazya2010} the following
two results on regularity shifts in corner-weighted spaces of finite order.
\begin{theorem}
\label{th:stokes-homogeneous-reg}
Let 
$\ell \geq 0$ be an integer, $1<s<\infty$, $\gamma \in {\mathbb R}$ 
and let
\begin{gather*}
    \bbf \in \bcK^{\ell, s}_{\gamma-2}(K), \qquad g\in \cK^{\ell+1, s}_{\gamma-1}(K).
  \end{gather*}
Suppose furthermore that
the line $\Re\lambda = \gamma-2/s$ does not contain eigenvalues of
$\mathfrak{A}(\lambda)$.
Then, there exist unique solutions $\bu$, $p$ to \eqref{eq:stokes}, 
with
\begin{equation*}
  \bu \in \bcK^{\ell+2,s}_{\gamma}(K), \qquad p\in \cK^{\ell+1,s}_{\gamma-1}(K)
\end{equation*}
and 
there exists $C>0$ (possibly dependent on $\ell$ but independent of $\bu$ and $p$) 
such that
\begin{equation}
  \label{eq:homog-estimate}
  \| \bu \|_{\bcK^{\ell+2,s}_\gamma(K)} +  \|p\|_{\cK^{\ell+1,s}_{\gamma-1}(K)} 
  \leq C 
  \left\{ \| \bbf \|_{\bcK^{\ell,s}_{\gamma-2}(K)} + \| g\|_{\cK^{\ell+1,s}_{\gamma-1}(K)} \right\}.
\end{equation}
\end{theorem}
\begin{theorem}
\label{th:stokes-homogeneous-reg-2}
Let $\ell \geq 0$ be an integer, $1<s<\infty$, and let $\bu, p$ be solutions to \eqref{eq:stokes}, 
where
\begin{equation*}
    \bbf \in \bcK^{\ell, s}_{\gamma-2}(K), \qquad g\in \cK^{\ell+1, s}_{\gamma-1}(K).
\end{equation*}
Suppose that
\begin{equation*}
  \bu \in \bcK^{\ell+1,s}_{\gamma}(K), \qquad  p \in \cK^{\ell,s}_{\gamma-1}(K), \qquad \eta
  \bu \in \bm{W}^{2,s}(K), \qquad \eta p \in W^{1,s}(K),
\end{equation*}
for all $\eta\in C^\infty_0(\overline{K}\setminus\{0\})$. 

Then $\bu\in \bcK^{\ell+2,s}_{\gamma}(K)$, $p\in \cK^{\ell+1,s}_{\gamma-1}(K)$ and 
there exists a constant $C>0$ (depending on $\ell$, $s$ and on $\gamma$) such that
 \begin{multline}
 \label{eq:homog-estimate-bis}
  \| \bu \|_{\bcK^{\ell+2,s}_\gamma(K)} +  \|p\|_{\cK^{\ell+1,s}_{\gamma-1}(K)} 
\\ \leq C \left\{ \| \bbf \|_{\bcK^{\ell,s}_{\gamma-2}(K)} + \| g\|_{\cK^{\ell+1,s}_{\gamma-1}(K)}
   + \| \bu \|_{\bcK^{\ell+1,s}_\gamma(K)} +  \|p\|_{\cK^{\ell,s}_{\gamma-1}(K)} \right\}.
\end{multline}
\end{theorem}

\begin{remark}
The theorems in \cite{Mazya2010} are stated with the operator  pencil
\begin{equation}
\begin{pmatrix}
    -\lambda^2 - \partial_\theta^2 & 0 & (\lambda -1 )\cos\theta-\sin\theta\partial_\theta \\
    0 & -\lambda^2 - \partial_\theta^2 & (\lambda -1)\sin\theta + \cos\theta\partial_\theta\\
    -(\lambda-1)\cos\theta+(\cos\theta)^2\partial_\theta 
    & (\lambda-1)\sin\theta+(\sin\theta)^2\partial_\theta&0
  \end{pmatrix}
\end{equation}
coming from the individual transformation in polar coordinates 
of the terms of the original matrix operator 
followed by multiplication by $r^2$ and substitution of $r\partial_r$ with $\lambda$. 
This symbol would therefore be applied to the Cartesian components $u_x$, $u_y$ of the Stokes flow.

The symbol \eqref{eq:L-symbol} results instead from the transformation of the 
Stokes equation into plane polar coordinates, i.e., 
it is applied to the polar components $u_r$ and $u_\theta$ of the velocity field $\bu$.
Since the Cartesian components $u_x$ and $u_y$ of $\bu$ 
can be obtained through a rotation of the polar components, the
eigenvalues of the symbol $\mathfrak{A}(\lambda)$ 
coincide \cite[Theorem 4.15]{Guo2006a}.
\end{remark}
Denote now by $\Lambda_{\min{}}$ the real part of the eigenvalue 
with smallest positive real part of the pencil $\mathfrak{A}(\lambda)$. 
Note that $\Lambda_{\min{}} > 1/2$ and that 
the strip $-\Lambda_{\min{}}<\Re\lambda< \Lambda_{\min{}}$ does
not contain eigenvalues of $\mathfrak{A}(\lambda)$.
Its existence is proved, for example, in \cite[Section 5.1]{Kozlov2001}.

\begin{remark}
Theorem \ref{th:stokes-homogeneous-reg} implies that, if 
$-\Lambda_{\min{}}< \gamma-2/s<\Lambda_{\min{}}$, 
$\bbf \in \bcK^{\infty, s}_{\gamma-2}(K)$, 
and
  $g\in \cK^{\infty, s}_{\gamma-1}(K)$
then
\begin{equation*}
\bu \in \bcK^{\infty, s}_{\gamma}(K)\qquad p\in \cK^{\infty,s}_{\gamma-1}(K).
\end{equation*}
\end{remark}
\begin{remark}
The $\ell = 0$ version of inequalities
\eqref{eq:homog-estimate} and \eqref{eq:homog-estimate-bis} read, respectively,
\begin{equation}
  \label{eq:homog-estimate-2}
\begin{multlined}[.9\displaywidth]
  \sum_{\alpham \leq 2} \| r^{\alpham-\gamma} \dalpha \bu \|_{L^s(K)} +
  \sum_{\alpham\leq 1} \|r^{\alpham-\gamma+1}\dalpha p\|_{L^s(K)}\\ \leq C
  \left\{ \|r^{2-\gamma} \bbf \|_{L^s(K)} + \sum_{\alpham \leq
      1}\|r^{\alpham-\gamma+1} \dalpha g\|_{L^s(K)}
      \right\}
\end{multlined}
\end{equation}
and
\begin{equation}
  \label{eq:homog-estimate-bis-2}
\begin{multlined}[.9\displaywidth]
  \sum_{\alpham \leq 2} \| r^{\alpham-\gamma} \dalpha \bu \|_{L^s(K)} +
  \sum_{\alpham\leq 1} \|r^{\alpham-\gamma+1}\dalpha p\|_{L^s(K)}\
  \\\qquad \leq C
  \bigg\{ \|r^{2-\gamma} \bbf \|_{L^s(K)} + \sum_{\alpham \leq
      1}\|r^{\alpham-\gamma+1} \dalpha g\|_{L^s(K)}
    + \sum_{\alpham\leq 1}\| r^{\alpham -\gamma}\dalpha\bu \|_{L^s(K)}
  \\
    +  \|r^{1-\gamma}p\|_{L^s(K)} 
      \bigg\}.
\end{multlined}
\end{equation}

\end{remark}
For a radius $0 < R < 1$, an integer $j$, and a scalar $\rho$ such that
$R-j\rho>0$, we introduce the sectors
\begin{equation}
  \label{eq:sectors}
  S_{R-j\rho} = B_{R-j\rho}\cap K = \left\{ x\in K: |x| < R-j\rho \right\},
\end{equation}
as shown in Figure \ref{fig:sectors}. 
\begin{figure}
  \centering
  \includegraphics[width=.45\textwidth]{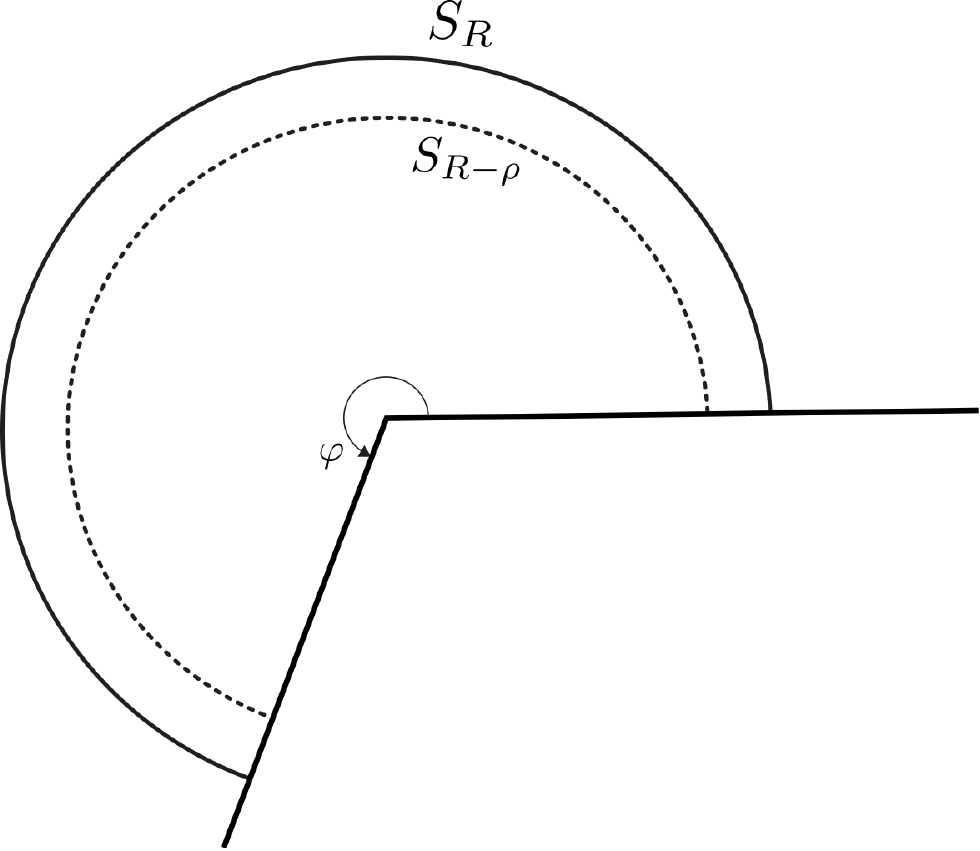}
  \caption{The sectors $S_{R}$, $S_{R-\rho}$.}
  \label{fig:sectors}
\end{figure}
Let $\chihat:\mathbb{R}\to[0,1]$ be a smooth step function such that
  \begin{equation*}
    \chihat\in C^\infty(\mathbb{R}), \qquad \chihat_{|_{(-\infty,0]}}=1, \qquad \chihat_{|_{[1,\infty)}} = 0 .
  \end{equation*}
  There exists a constant $C_\chi \geq 1$ such that, for all $x\in \mathbb{R}$,
  \begin{equation}
    \label{eq:chihat-def}
    |(\dalpha \chihat) (x)|\leq C_\chi,\quad \forall \alpha\in \mathbb{N}_0\text{ such that } \alpham \leq 2.
  \end{equation}
For all $j\in\mathbb{N}_0$ and all $\rho>0$ such that $R-j\rho \geq R/2$, we can then introduce a smooth cutoff function
$\chi:\mathbb{R}^2\to[0,1]$, such that $\chi\in C^\infty_0(B_{R-j\rho})$,
defined, for all $x\in \mathbb{R}^2$, as 
\begin{equation*}
  \chi(x) =
  \begin{cases}
    1 &\text{if }|x|\leq R-(j+1)\rho,\\
    \chihat\left(\dfrac{|x|-(R-(j+1)\rho)}{\rho}\right)&\text{if }R-(j+1)\rho<|x|<R-j\rho,\\
    0 &\text{if }|x|\geq R-j\rho.
  \end{cases}
\end{equation*}
By algebraic manipulations, \eqref{eq:chihat-def}, and since $R-j\rho\geq R/2$, it can
be shown that for all $x\in \mathbb{R}^2$ and all $\alpha\in \mathbb{N}^2_0$
with $\alpham \leq 2$, there holds
\begin{equation}
  \label{eq:chidef}
  \left| (\dalpha\chi)(x)\right|\leq \left( \frac{4}{R}+1 \right) C_\chi \rho^{-|\alpha|}.
\end{equation}We introduce several auxiliary lemmas.
\begin{lemma}
  \label{lemma:comm1}
  Let $\ell, j\in \mathbb{N}_0$, $1<s<\infty$, $\gamma \in \mathbb{R}$,
  $0<R_0\leq R<1$, 
  $\rho \in \mathbb{R}$ such that $R-j\rho \geq R/2$, and 
$v\in \cK^{\ell,s}_\gamma(S_{R-j\rho})$.
Then there exists $C>0$ such that, for any $\betam =
\ell$,
  \begin{equation*}
    \sum_{\alpham=1}\| r^{1-\gamma}\left[ \chi r^\ell , \dalpha \right] \dbeta v\|_{L^s(S_{R-(j+1)\rho})}
    \leq
    C(\ell + \rho^{-1}) \|r^{\ell-\gamma} \dbeta v\|_{L^s(S_{R-j\rho})},
  \end{equation*}
with $C$ dependent only on $C_\chi$ in \eqref{eq:chidef} and on $R_0$.
\end{lemma}
\begin{proof}
  For all $x\in \mathbb{R}^2$ and $\alpham = 1$, there holds $|\dalpha
  (r(x)^\ell)|\leq \ell r(x)^{\ell-1}$. By elementary manipulations,
  we have, for all multiindices $\betam = \ell$ and $\alpham=1$,
  \begin{multline*}
    \| r^{1-\gamma}\left[ \chi r^\ell , \dalpha \right] \dbeta v\|_{L^s(S_{R-(j+1)\rho})}
    \\ \leq 
    \| \dalpha \chi \|_{L^\infty(S_R)}\| r^{\ell+1-\gamma}\dbeta v\|_{L^s(S_{R-j\rho})} 
        + \ell \|r^{\ell-\gamma}\dbeta v\|_{L^s(S_{R-j\rho})}.
  \end{multline*}
  Using \eqref{eq:chidef} gives the assertion.
\end{proof}
\begin{lemma}
  \label{lemma:comm2}
  Let $\ell, j\in \mathbb{N}_0$, $1<s<\infty$, $\gamma \in \mathbb{R}$,
  $0<R_0\leq R<1$, $\rho \in
  \mathbb{R}$ such that $R-j\rho \geq R/2$, and $v\in
  \cK^{\ell+1,s}_\gamma(S_{R-j\rho})$. Then there exists
  $C>0$ such that, for any $\betam = \ell$,
  \begin{multline*}
    \sum_{\alpham=2}    \| r^{2-\gamma}\left[ \chi r^\ell , \dalpha \right] \dbeta v\|_{L^s(S_{R-(j+1)\rho})}
\\
    \leq
    C\sum_{\deltam=0,1}\left(\ell + \rho^{-1}\right)^{2-\deltam} \|r^{\ell+\deltam-\gamma} \partial^{\beta+\delta} v\|_{L^s(S_{R-j\rho})}.
  \end{multline*}
The constant $C$ depends only on $C_\chi$ in \eqref{eq:chidef} and on
  $R_0$.
\end{lemma}
\begin{proof}
  The proof follows the same lines as the proof of the last lemma. 
  We have, for all multiindices $\betam = \ell$ and $\alpham=2$
  \begin{align*}
   & \| r^{2-\gamma}\left[ \chi r^\ell , \dalpha \right] \dbeta v\|_{L^s(S_{R-(j+1)\rho})}
     \\
   &  \qquad \leq
      \| r^{2-\gamma} \dalpha \left( \chi r^\ell \right) \dbeta v\|_{L^s(S_{R-(j+1)\rho})}
      +
      2\sum_{\deltam = 1:\delta\leq \alpha}\| r^{2-\gamma} \ddelta \left( \chi r^\ell \right) \partial^{\beta + \alpha -\delta} v\|_{L^s(S_{R-(j+1)\rho})}\\
    & \qquad\leq
    \left( C \rho^{-2} + C \ell \rho^{-1} + \ell^2 \right) \|r^{\ell-\gamma} \dbeta v\|_{L^s(S_{R-j\rho})}
    \\
    & \qquad \qquad +
     2\sum_{\deltam = 1 :\delta\leq\alpha} \left( C\rho^{-1}+\ell \right)\| r^{\ell+1-\gamma}  \partial^{\beta + \alpha -\delta} v\|_{L^s(S_{R-j\rho})},
  \end{align*}
  hence the assertion follows.
\end{proof}
We now give a localised and explicit version of Theorem \ref{th:stokes-homogeneous-reg}.
\begin{proposition}
  \label{th:stokes-estimate}
 Let $\ell\in \mathbb{N}$, $1<s<\infty$, $0< R<1$, and $-\Lambda_{\min{}}<\gamma-2/s < \Lambda_{\min{}}$.
 Suppose that 
\begin{equation}
  \label{eq:k-hypothesis}
  \bbf \in \bcK^{\ell,s}_{\gamma-2}(K),
  \qquad g\in
  \cK^{\ell+1,s}_{\gamma-1}(K) .
\end{equation}
Then, $\bu$, $p$ are the unique solutions to the Stokes system \eqref{eq:stokes}, 
with $\bu \in \bcK^{\ell+2,s}_{\gamma}(K)$, $p\in \cK^{\ell+1,s}_{\gamma-1}(K)$, 
and there exists $C>0$ such that for all 
$j\in \mathbb{N}$ such that $j\geq \ell$, and for all $\rho \in (0, R/(2j)]$,
\begin{multline}
  \label{eq:third-weighted}
  | \bu |_{\bcK^{\ell+2,s}_\gamma (S_{R-(j+1)\rho})}
  +     | p |_{\cK^{\ell+1,s}_{\gamma-1}(S_{R-(j+1)\rho})} 
  \leq 
C \bigg\{
    |  \bbf |_{\bcK^{\ell,s}_{\gamma-2}(S_{R-j\rho})}
    +  \sum_{n\leq 1}|  g
    |_{\cK^{\ell+n,s}_{\gamma-1}(S_{R-j\rho})} \\
  +\rho^{-1}  |  p |_{\cK^{\ell,s}_{\gamma-1}(S_{R-j\rho})}
  + \sum_{n \leq 1}\rho^{n-2}|  \bu |_{\bcK^{\ell+n,s}_{\gamma}(S_{R-j\rho})}
  \bigg\}.
\end{multline}
\end{proposition}
\begin{proof}
Fix $\ell\in \mathbb{N}$, the integer $j\geq \ell$, and assume \eqref{eq:k-hypothesis}.
Then it follows from Theorem \ref{th:stokes-homogeneous-reg} that
\begin{equation}
  \label{eq:first-reg}
  \bu \in \bcK^{\ell+2,s}_\gamma(K), \qquad p \in \cK^{\ell+1, s}_{\gamma-1}(K).
\end{equation}
Consider $\beta\in {\mathbb N}_0^2$ 
such that $\betam = \ell$ and write $\bw = \chi r^\betam\dbeta \bu$ 
and $q = \chi r^\betam\dbeta p$.
Then,
\begin{align*}
  - \Delta \bw + \nabla q
  & = \left[ -\Delta, \chi r^\betam \right] \dbeta\bu
    + \left[ \nabla, \chi r^\betam \right] \dbeta p  + \chi r^\betam \dbeta (-\Delta \bu + \nabla p)
  \\
  & = \left[ -\Delta, \chi r^\betam \right] \dbeta\bu
    + \left[ \nabla, \chi r^\betam \right] \dbeta p  + \chi r^\betam \dbeta \bbf
\end{align*}
hence \eqref{eq:k-hypothesis}, \eqref{eq:first-reg}, and Lemmas
\ref{lemma:comm1} and \ref{lemma:comm2} give
$-\Delta \bw + \nabla q \in
\bcK^{0,s}_{\gamma-2}(K)$. Then,
\begin{equation*}
  \nabla \cdot \bw
   = \left( \nabla (\chi r^\betam) \right) \cdot (\dbeta \bu) + \chi r^\betam \dbeta g
   \in \cK^{1,s}_{\gamma-1}(K).
\end{equation*}

We can therefore apply Theorem
\ref{th:stokes-homogeneous-reg} with
$\ell = 0$. Inequality \eqref{eq:homog-estimate-2} implies that there exists
$C>0$ independent of $\bw$ and $q$ such that
\begin{multline}
  \label{eq:first-weighted}
  \sum_{\alpham \leq 2} \| r^{\alpham-\gamma} \dalpha \bw \|_{L^s(K)} +
  \sum_{\alpham\leq 1} \|r^{\alpham-\gamma+1}\dalpha q\|_{L^s(K)}
  \\ \leq C \bigg\{ \|r^{2-\gamma} \left( -\Delta \bw + \nabla q \right)
    \|_{L^s(K)} + \sum_{\alpham \leq 1}\|r^{\alpham-\gamma+1} \dalpha\nabla\cdot
    \bw\|_{L^s(K)} 
    \bigg\}.
\end{multline}
By algebraic manipulation and the triangle inequality, the following
inequalities hold, for $\betam = \ell$
\begin{multline}
  \label{eq:weighted-comm-1}
  \sum_{\alpham = 2} \| r^{\betam+2-\gamma}\partial^{\alpha + \beta} \bu \|_{L^s(S_{R-(j+1)\rho})} \\
  \leq
  \sum_{\alpham = 2} \|r^{2-\gamma} \left[ \chi r^\betam, \dalpha \right] \dbeta \bu \|_{L^s(S_{R-(j+1)\rho})} + 
  \sum_{\alpham = 2} \|r^{2-\gamma}  \dalpha \bw  \|_{L^s(S_{R-(j+1)\rho})} 
\end{multline}
and
\begin{multline}
  \label{eq:weighted-comm-2}
  \sum_{\alpham = 1} \| r^{\betam+2-\gamma}\partial^{\alpha + \beta} p \|_{L^s(S_{R-(j+1)\rho})} \\
  \leq
  \sum_{\alpham = 1} \|r^{2-\gamma} \left[ \chi r^\betam, \dalpha \right] \dbeta p \|_{L^s(S_{R-(j+1)\rho})} + 
  \sum_{\alpham = 1} \|r^{2-\gamma}  \dalpha q \|_{L^s(S_{R-(j+1)\rho})} .
\end{multline}
We can estimate the terms with the commutators in \eqref{eq:weighted-comm-1} and
\eqref{eq:weighted-comm-2}: Lemmas \ref{lemma:comm1} and \ref{lemma:comm2} give
\begin{multline}
  \label{eq:commut-1}
  \sum_{\alpham = 2} \|r^{2-\gamma} \left[ \chi r^\betam, \dalpha \right] \dbeta \bu \|_{L^s(S_{R-(j+1)\rho})} 
\\
  \leq
 C \sum_{\alpham\leq 1}\left( \betam + \rho^{-1} \right)^{2-\alpham} \| r^{\betam +\alpham-\gamma} \partial^{\beta+\alpha} \bu \|_{L^s(S_{R-j\rho})},
\end{multline}
and
\begin{equation}
  \label{eq:commut-2}
  \sum_{\alpham = 1} \|r^{2-\gamma} \left[ \chi r^\betam, \dalpha \right] \dbeta p \|_{L^s(S_{R-(j+1)\rho})} 
  \leq
 C \left( \betam + \rho^{-1} \right) \| r^{\betam -\gamma+1} \dbeta p \|_{L^s(S_{R-j\rho})}
\end{equation}
with constant $C$ independent of $\betam$, $\rho$, $\bu$, and $p$.

Now, since $\rho \leq R/(2j)$ and $j\geq \ell$, then $\betam \leq \rho^{-1}$. 
In addition, from the inequalities \eqref{eq:weighted-comm-1}--\eqref{eq:commut-2} and using
\eqref{eq:first-weighted}, we obtain 
\begin{multline}
  \label{eq:second-weighted}
  \sum_{\alpham = 2} \| r^{\betam+\alpham-\gamma} \partial^{\beta+\alpha} \bu \|_{L^s(S_{R-(j+1)\rho})} +
  \sum_{\alpham = 1} \|r^{\betam+\alpham-\gamma+1}\partial^{\beta+\alpha} p\|_{L^s(S_{R-(j+1)\rho})}\\
\leq
  C \bigg\{ \|r^{2-\gamma} \left( -\Delta \bw + \nabla q \right)
    \|_{L^s(S_{R-j\rho})} + \sum_{\alpham \leq 1}\|r^{\alpham-\gamma+1} \dalpha\nabla\cdot \bw\|_{L^s(S_{R-j\rho})} \\
 +\rho^{-1}  \| r^{\betam -\gamma+1} \dbeta p \|_{L^s(S_{R-j\rho})}
 + \sum_{\alpham\leq 1} \rho^{\alpham-2} \| r^{\betam +\alpham-\gamma} \partial^{\beta+\alpha} \bu \|_{L^s(S_{R-j\rho})}
  \bigg\}.
\end{multline}
Using the commutators as before, the definition of $\bw$ and $q$, Lemmas
\ref{lemma:comm1} and \ref{lemma:comm2}, and the
inequality $\betam \leq \rho^{-1}$ we can show that
\begin{align*}
  & \|r^{2-\gamma} \left( -\Delta \bw + \nabla q \right) \|_{L^s(S_{R-j\rho})}
    \\
  &\qquad \leq 
    \begin{multlined}[t][.8\textwidth]
  \|r^{2-\gamma} \left[  \Delta , r^\betam  \right] \dbeta \bu \|_{L^s(S_{R-j\rho})}
  + \|r^{2-\gamma} \left[  \nabla, r^\betam  \right] \dbeta p \|_{L^s(S_{R-j\rho})}
  \\
  + \|r^{\betam-\gamma+1} \dbeta \left( -\Delta \bu + \nabla p \right)\|_{L^s(S_{R-j\rho})}
    \end{multlined}
   \\ & \qquad \leq
    \begin{multlined}[t][.8\textwidth]
  C \sum_{\alpham\leq 1} \rho^{\alpham-2} \| r^{\betam +\alpham-\gamma} \partial^{\beta+\alpha} \bu \|_{L^s(S_{R-j\rho})}
 + C\rho^{-1}  \| r^{\betam -\gamma+1} \dbeta p \|_{L^s(S_{R-j\rho})}
 \\
  + \|r^{2+\betam-\gamma} \dbeta \bbf \|_{L^s(S_{R-j\rho})}
\end{multlined}
\end{align*}
and, for all $\alpham \leq 1$
\begin{align*}
  &\|r^{2-\gamma}\dalpha \nabla\cdot \bw\|_{L^s(S_{R-j\rho})} 
  \\
  & \qquad \leq
  \sum_{\deltam = 1}\|r^{2-\gamma} \dalpha\left[ \partial^\delta , r^\betam  \right] \dbeta \bu \|_{L^s(S_{R-j\rho})}
 + \| r^{\betam -\gamma+2} \partial^{\beta+\alpha} \nabla\cdot \bu \|_{L^s(S_{R-j\rho})}
\\ &\qquad  \leq
  C \rho^{-1}\|r^{\betam-\gamma+1} \partial^{\beta+\alpha} \bu \|_{L^s(S_{R-j\rho})}
 + \| r^{\betam -\gamma+2} \partial^{\beta+\alpha} g \|_{L^s(S_{R-j\rho})}.
\end{align*}
Summing \eqref{eq:second-weighted} over all $\betam = \ell$ and using the last
two inequalities, 
\begin{multline*}
  | \bu |_{\bcK^{\ell+2,s}_\gamma (S_{R-(j+1)\rho})}
  +     | p |_{\cK^{\ell+1,s}_{\gamma-1}(S_{R-(j+1)\rho})} 
  \leq 
C \bigg\{
    |  \bbf |_{\bcK^{\ell,s}_{\gamma-2}(S_{R-j\rho})}
    +  \sum_{n\leq 1}|  g
    |_{\cK^{\ell+n,s}_{\gamma-1}(S_{R-j\rho})} \\
  +\rho^{-1}  |  p |_{\cK^{\ell,s}_{\gamma-1}(S_{R-j\rho})}
  + \sum_{n \leq 1}\rho^{n-2}|  \bu |_{\bcK^{\ell+n,s}_{\gamma}(S_{R-j\rho})}
  \bigg\}.
\end{multline*}
\end{proof}
\begin{proposition}
  \label{th:stokes-estimate-bis}
  Let $\ell\in \mathbb{N}_0$, $\gamma \in \mathbb{R}$, $1<s<\infty$,
  $0< R<1$, $j\in \mathbb{N}_0$ such that
  $j\geq \ell$, and $\rho \in (0,
R/(2j)]$.
Let $\bu$, $p$ be solutions to the Stokes system \eqref{eq:stokes}, with
\begin{equation}
  \label{eq:k-hypothesis-bis}
  \bbf \in \bcK^{\ell,s}_{\gamma-2}(S_{R-j\rho}),
  \qquad g\in
  \cK^{\ell+1,s}_{\gamma-1}(S_{R-j\rho}).
\end{equation}
Suppose that 
\begin{equation*}
  \bu \in \bcK^{\ell+1,s}_{\gamma}(S_{R-j\rho}), \quad  p \in \cK^{\ell,s}_{\gamma-1}(S_{R-j\rho}), \quad \eta
  \bu \in \bm{W}^{2,s}(S_{R-j\rho}), \quad \eta p \in W^{1,s}(S_{R-j\rho}),
\end{equation*}
for all 
$\eta \in C^\infty_0(\bar{S}_{R-j\rho}\setminus\{0\})$.
Then, $\bu \in \bcK^{\ell+2,s}_{\gamma}(S_{R-(j+1)\rho})$, $p\in
\cK^{\ell+1,s}_{\gamma-1}(S_{R-(j+1)\rho})$, and there exists $C>0$ independent
of $\ell$, $j$, $\rho$, such that \eqref{eq:third-weighted} holds.
\end{proposition}
\begin{proof}
The proof follows the same steps as the proof of Proposition
\ref{th:stokes-estimate}, with \eqref{eq:homog-estimate-bis-2} replacing
\eqref{eq:homog-estimate-2} in \eqref{eq:first-weighted}.
\end{proof}
\subsection{Analytic regularity shift for the quadratic nonlinearity in a sector}
\label{sec:nonlin-term}
To handle the regularity shift for the quadratic nonlinearity 
   in the weighted scales $\cK^{\ell,s}_{\gamma}$,
we apply a technique inspired by \cite{DallAcqua2012} and used also in \cite{nonlinear},
for different problems but similar nonlinearities. 
We derive interpolation inequalities in weighted spaces and 
use them to bound the quadratic nonlinearity
in the momentum transport term of the Navier-Stokes system in 
the homogeneous, corner-weighted norms.
\subsubsection{Induction hypotheses}
\label{sec:IndHyp}
For 
$1 <s <\infty$, $\gamma \in \mathbb{R}$, $j\in \mathbb{N}_0$, 
$0< R< 1$ and $C_{\bu}, A_{\bu}, C_p,A_p>0$,
we introduce two statements on the boundedness of high order derivatives of
weak solutions with quantitative control of their size in weighted
Kondrat'ev spaces. 
We shall say that
\emph{$H_{\bu}(s, \gamma, j, C_{\bu}, A_{\bu})$ holds in $S_R$} 
if 
for all $\rho\in (0, R/(2j)]$,
  \begin{subequations}
  \begin{equation}
    \label{eq:induction-u}
    \left| \bu \right|_{\bcK^{\ell,s}_\gamma(S_{R-j\rho})} 
     \leq C_{\bu} A_{\bu}^\ell (j\rho)^{-\ell} \ell^\ell \text{ for all }\ell =0, \dots, j
  \end{equation}
  and we shall say that \emph{$H_p(s,\gamma,  j, C_p, A_p)$ holds in $S_R$} 
  if for all $\rho\in (0, R/(2j)]$,
  \begin{equation}
    \label{eq:induction-p}
    \left| p \right|_{\cK^{\ell-1,s}_{\gamma-1}(S_{R-j\rho})} 
    \leq C_{p} A_{p}^\ell (j\rho)^{-\ell} \ell^\ell \text{ for all }\ell=1, \dots, j.
  \end{equation} 
\end{subequations}
For $k\in \mathbb{N}$, we state the induction hypothesis succinctly as
\begin{equation}
  \label{eq:induction}
  H_{\bu,p}(s,\gamma,k,C_{\bu,p}, A_{\bu,p}):\quad
    \begin{cases}
H_{\bu}(s,\gamma, j,C_{\bu}, A_{\bu}) 
\text{ holds for }j=0, \dots, k
\\
H_p(s,\gamma, j, C_p, A_p)\text{ holds for }j=1, \dots, k.
  \end{cases}
\end{equation}

\subsubsection{Weighted interpolation estimates}
\label{sec:Weighted}
We introduce interpolation estimates for weighted norms in sectors
near the origin, analogous to the estimates used to establish regularity
in classical Sobolev spaces, see, e.g., \cite{Adams2003}. 
The derivation of these estimates is 
based on the homogeneity of the weighted norms (which is
a consequence of the homogeneous Dirichlet boundary conditions), 
a dyadic decomposition of the domain around the point singularity
combined with an homothety to a reference annulus of unit size.

\begin{lemma}
  \label{lemma:thetaprod}
  Let  $0 <D_0 \leq D \leq 1$ and let $S_D$ be defined as in \eqref{eq:sectors}, $\beta \in \mathbb{N}_0^2$ such
  that $\betam > 0$, $v\in \mathcal{K}^{\betam+1,
    s}_\gamma(S_D)$, $\gamma-2/s > \xi -2/t$, and $2< s \leq t \leq \infty$. Then, the following ``interpolation'' estimate holds
  \begin{equation}
  \label{eq:thetaprod}
  \|r^{-\xi+|\beta|}\dbeta v \|_{L^{t}(S_D)} 
  \leq 
  C\|r^{\betam-\gamma} \dbeta v\|^{1-\theta}_{L^s(S_D)}
  \sum_{\alpham \leq 1}  
  \betam^{(1-\alpham)\theta}\| r^{\betam+\alpham-\gamma} \partial^{\beta+\alpha} v \|^\theta_{L^s(S_D)}
  \end{equation}
  with $\theta = 2/s - 2/t$
  and for $C$
  dependent on $s$, $t$, $\gamma$, $\xi$ and $D_0$, but independent of $\betam$, and of $v$.
\end{lemma}
\begin{proof}
  We consider the case where $D= 1$, the general one will follow by scaling and
  the fact that $D\geq D_0$. For $j\in \mathbb{N}_0$, we
  introduce the dyadic annuli
  \begin{equation*}
    V^j = \left\{ x\in\mathbb{R}^2: 2^{-j-1} < |x| < 2^{-j} \right\} 
  \end{equation*}
  and their sectorial intersections
  \begin{equation*}
    S^j = S_1 \cap V^j.
  \end{equation*}
We also introduce the homothetic dilation $\Phi_j: V^j \to V^0: x \mapsto 2^jx$
and denote with a hat the scaled quantities, e.g., $\hv = v \circ \Phi_j^{-1}$.
Then, for all $j\in \mathbb{N}$,
\begin{equation}
  \label{eq:interp-proof1}
   \|r^{-\xi+|\beta|}\dbeta v \|_{L^{t}(S^j)}
   \leq
   2^{-j(-\xi+2/t)}\|\hr^{-\xi+|\beta|}\hdbeta \hv \|_{L^{t}(S^0)}.
\end{equation}
On $S^0$ we have $1/2 < \hr < 1$, then, 
\begin{equation}
  \label{eq:interp-proof2}
  \|\hr^{-\xi+|\beta|}\hdbeta \hv \|_{L^{t}(S^0)}
  \leq
  2^{-|\xi|}\|\hr^{|\beta|}\hdbeta \hv \|_{L^{t}(S^0)}.
\end{equation}
Furthermore, since $S^0$ satisfies the cone condition \cite{Adams2003}, 
there holds the interpolation estimate
\begin{equation*}
  \| v \|_{L^{t}(S^0)} \leq C \| v \|^\theta_{L^s(S^0)} \| v \|^{1-\theta}_{W^{1,s}(S^0)},
\end{equation*}
valid for all $v\in W^{1,s}(S^0)$, with $\theta = 2/s - 2/t \in [0,1)$. 
From \eqref{eq:interp-proof1}, \eqref{eq:interp-proof2} and 
from this interpolation estimate we obtain
\begin{equation*}
   \|r^{-\xi+|\beta|}\dbeta v \|_{L^{t}(S^j)}
   \leq
   C2^{-j(-\xi+2/t)}\|\hr^{|\beta|}\hdbeta \hv \|_{L^{ s}(S^0)}^{1-\theta}
   \sum_{\alpham \leq 1}\|\hdalpha \hr^{|\beta|}\hdbeta \hv \|_{L^{ s}(S^0)}^\theta.
\end{equation*}
By an elementary manipulation, using again the fact that $1/2 < \hr < 1$,
and possibly adjusting the constant $C$,
\begin{multline} \label{eq:interp-proof4}
   \|r^{-\xi+|\beta|}\dbeta v \|_{L^{t}(S^j)} \\ \leq 
   C2^{-j(\xi+2/t)}\|\hr^{|\beta|-\gamma}\hdbeta \hv \|_{L^{ s}(S^0)}^{1-\theta}
   \sum_{\alpham \leq 1}
   \betam^{(1-\alpham)\theta}\| \hr^{|\beta|+\alpham -\gamma}\hdalpha\hdbeta \hv \|_{L^{ s}(S^0)}^\theta.
\end{multline}
Scaling back from $S^0$ to $S^j$,
\begin{multline*}
  \|r^{-\xi+|\beta|}\dbeta v \|_{L^{t}(S^j)}
  \\
  \leq 
   C 2^{-j(-\xi+2/t + \gamma-2/s)}
\|r^{|\beta|-\gamma}\dbeta v \|_{L^{ s}(S^j)}^{1-\theta}
  \sum_{\alpham \leq 1}\betam^{(1-\alpham)\theta}\| 
   r^{|\beta|+\alpham -\gamma}\dalpha\dbeta v \|_{L^{ s}(S^j)}^\theta.
\end{multline*}
If $\gamma -2/s > \xi -2/t$, then, we can sum over all $S^j$ and obtain the assertion.
\end{proof}
\begin{remark}
  The constant $C$ in \eqref{eq:thetaprod} depends on $s$, $t$, $\gamma$, and $\xi$. 
  Nonetheless,
  it will be used only for specific values of $s$, $t$, $\gamma$, and $\xi$, thus
  this dependence will not be relevant in the following proofs. 
  Specifically,
  under the hypothesis of Lemma \ref{lemma:thetaprod}, we will use the inequalities
  \begin{equation}
  \label{eq:thetaprod-2s}
  \|r^{\frac{1-\gamma}{2}+|\beta|}\dbeta v \|_{L^{2s}(S_D)} \leq C\|r^{\betam-\gamma} \dbeta v\|^{1-\theta}_{L^s(S_D)}
  \sum_{\alpham \leq 1}  \betam^{(1-\alpham)\theta}\| r^{\betam+\alpham-\gamma} \partial^{\beta+\alpha} v \|^\theta_{L^s(S_D)}
  \end{equation}
   valid for $\gamma - 2/s>-1$ and with $\theta  = 1/s$, and
\begin{equation}
  \label{eq:thetaprod-infty}
  \|r^{1+|\beta|}\dbeta v \|_{L^{\infty}(S_D)} \leq C\|r^{\betam-\gamma} \dbeta v\|^{1-\eta}_{L^s(S_D)}
  \sum_{\alpham \leq 1}  \betam^{(1-\alpham)\eta}\| r^{\betam+\alpham-\gamma} \partial^{\beta+\alpha} v \|^\eta_{L^s(S_D)},
  \end{equation}
  valid for $\gamma -2/s > -1$ and with $\eta = 2/s$.
\end{remark}
\subsubsection{Estimate of the nonlinear term}
We can now give an estimate of the norm of the nonlinear term, under the
assumption that the induction hypothesis holds. 
Specifically, this will be done in Lemma \ref{lemma:nonlin-est}. 
We start by introducing three auxiliary lemmas,
that will subsequently be necessary for the proof of Lemma \ref{lemma:nonlin-est}.
\begin{lemma}
  \label{lemma:nonlin1}
  Let $k\in \mathbb{N}$, $0< R<1$, and $2<s<\infty$, $\gamma$ such that $ \gamma -
  2/s>-1$.
    Let $\bu$ be such that $H_{\bu}(s, \gamma, k, C_{\bu}, A_{\bu})$ holds
    in $S_R$.
 Then, there exists $C>0$ independent of $k$, $A_{\bu}$, and $\rho$ such that
there holds, for all integer $0\leq \ell \leq k-1$ and with $\theta = 1/s$,
\begin{equation}
\sum_{\betam=\ell}\| r^{\frac{1-\gamma}{2}+\betam} \dbeta \bu \|_{L^{2 s} (S_{R-k\rho})} 
\leq 
C A_{\bu}^{\ell+\theta} (k\rho)^{-\ell-\theta} (\ell+1)^{\ell + \theta}
\;.
\end{equation}
\end{lemma}
\begin{proof}
  Let $\bu = (u_1, u_2)$ and denote $v = u_i$, for any $i=1,2$.
 We start by applying Lemma \ref{lemma:thetaprod}, 
 using the bound \eqref{eq:thetaprod-2s}: 
 \begin{multline}
   \label{eq:nonlin1-proof1}
   \|r^{\frac{1-\gamma}{2}+|\beta|}\dbeta v \|_{L^{2s}(S_{R-k\rho})}
   \\
   \leq
   C\|r^{\betam-\gamma} \dbeta v\|^{1-\theta}_{L^s(S_{R-k\rho})}
  \sum_{\alpham \leq 1}  \betam^{(1-\alpham)\theta}
      \| r^{\betam+\alpham-\gamma} \partial^{\beta+\alpha} v \|^\theta_{L^s(S_{R-k\rho})},
 \end{multline}
with $C$ independent of $\betam$ and $v$.
 Then, from \eqref{eq:induction-u} we have that for all $1\leq \ell\leq k-1$
 \begin{equation*}
\sum_{\betam=\ell}   \|r^{\ell-\gamma} \dbeta v\|^{1-\theta}_{L^s(S_{R-k\rho})} 
      \leq C_{\bu} A_{\bu}^{(1-\theta)\ell} (k\rho)^{-(1-\theta)\ell}\ell^{(1-\theta)\ell}
 \end{equation*}
 and 
 \begin{multline*}
\sum_{\betam=\ell}  \sum_{\alpham \leq 1}  
   \ell^{(1-\alpham)\theta}\| r^{\ell+\alpham-\gamma} \partial^{\beta+\alpha} v \|^\theta_{L^s(S_{R-k\rho})}
\\
  \leq
  C_{\bu} A_{\bu}^{\theta\ell} (k\rho)^{-\theta\ell}\ell^{\theta(\ell+1)} +
C_{\bu}  A_{\bu}^{\theta(\ell+1)} (k\rho)^{-\theta(\ell+1)}(\ell+1)^{\theta(\ell+1)}, 
 \end{multline*}
 hence
 \begin{equation*}
\sum_{\betam=\ell}   \|r^{\frac{1-\gamma}{2}+\ell}\dbeta v \|_{L^{2s}(S_{R-k\rho})}
   \leq
  C A_{\bu}^{\ell + \theta} (k\rho)^{-\ell - \theta}(\ell+1)^{\ell+\theta},
 \end{equation*}
 which is what was to be proved.
\end{proof}
\begin{lemma}
    \label{lemma:nonlin2}
    Under the same hypotheses as in Lemma \ref{lemma:nonlin1},
 there exists $C>0$ independent of $k$, $A_{\bu}$,  and $\rho$ such that 
      there holds for all $0\leq \ell\leq k-2$ and with $\theta = 1/s$
    \begin{equation}
   \sum_{\betam=\ell}\| r^{\frac{1-\gamma}{2}+1+\betam} \dbeta \nabla \bu \|_{L^{2 s} (S_{R-k\rho})} 
\leq C A_{\bu}^{\ell+1+\theta} (k\rho)^{-\ell-1-\theta} (\ell+2)^{\ell+1+\theta}
    \;.
\end{equation}
\end{lemma}
\begin{proof}
  The assertion follows directly from Lemma \ref{lemma:nonlin1}.
\end{proof}
\begin{lemma}
  \label{lemma:nonlininfty}
     Under the same hypotheses as in Lemma \ref{lemma:nonlin1},
 there exists $C>0$ independent of $k$, $A_{\bu}$, and $\rho$ such that 
      there holds, for all $0\leq \ell\leq k-1$ and with $\eta  =2/s$,
    \begin{equation}
    \sum_{\betam=\ell}\| r^{1+\betam} \dbeta \bu \|_{L^{\infty} (S_{R-k\rho})} 
    \leq 
     C A_{\bu}^{\ell+\eta} (k\rho)^{-\ell-\eta} (\ell+1)^{\ell+\eta}
     \;.
\end{equation}
\end{lemma}
\begin{proof}
  The proof follows verbatim the steps as the proof of Lemma
  \ref{lemma:nonlin1}, with \eqref{eq:nonlin1-proof1} replaced by 
  \eqref{eq:thetaprod-infty} and $\theta$ replaced by $\eta$.
\end{proof}
The weighted seminorms of the nonlinear term in \eqref{eq:ns-P} can now be
controlled, subject to the induction hypothesis.
\begin{lemma}
  \label{lemma:nonlin-est}
  Let $0<R<1$, $k\in \mathbb{N}$, $2<s<\infty$, $\gamma \in \mathbb{R}$, 
  $\bu\in\mathcal{K}^{k}_\gamma(S_{R-k\rho})$ be 
  such that $\gamma-2/s>-1$ and assume that 
  $H_{\bu}(s,\gamma,k,C_{\bu},A_{\bu})$ holds in $S_R$.

  Then,
  $\left( \bu\cdot \nabla \right) \bu \in \cK^{k-1, s}_{\gamma-2}(S_{R-k\rho})$
  and there exists a constant $C>0$ independent of $k$ and $A_{\bu}$, such that 
  for all $\rho\in (0, R/(2k)]$ and for all $\ell\in \mathbb{N}_0$ such that $\ell \leq k-1$ 
  there holds the bound 
  \begin{equation}
  \label{eq:nonlin-est}
\sum_{\betam=\ell}
    \| r^{2-\gamma+\betam} \dbeta \left( \bu\cdot \nabla \right) \bu \|_{L^s(S_{R-k\rho})} \leq
    C A_{\bu}^{\ell+1+2\theta} (k\rho)^{-\ell-1-2\theta} (\ell+1)^{\ell+3/2}
  \end{equation}
  with $\theta = 1/s$.
\end{lemma}
\begin{proof}
Let $\bu = (u_1, u_2)$ and $\ell\in \mathbb{N}$,
$\beta\in\mathbb{N}^2_0$ such that
$\betam = \ell \leq k-1$.
By the Leibniz rule and the H\"older inequality, for $m,n\in \{1,2\}$,
\begin{equation}
  \label{eq:ugu-proof1}
\begin{split}
  & \| r^{2-\gamma+\betam} \dbeta u_m\partial_m u_n \|_{L^s(S_{R-k\rho})}
  \\
  & \quad \leq \sum_{j = 0}^{\betam-1} \sum_{\substack{\alpham = j\\\alpha\leq\beta}}
    \binom{\beta}{\alpha} \|r^{\frac{1-\gamma}{2}+\betam-\alpham}\partial^{\beta-\alpha} u_m \|_{L^{2 s}(S_{R-k\rho})}\|r^{\frac{1-\gamma}{2}+
                     1+\alpham}\dalpha \partial_mu_n \|_{L^{2 s}(S_{R-k\rho})}
  \\
& \qquad \qquad
  + \| r^{2-\gamma +\betam }u_m \dbeta \partial_m u_n\|_{L^s(S_{R-k\rho})} .
\end{split}
\end{equation}
We start by considering the sum in the right hand side of the above inequality.
In the inequality \eqref{eq:product-inequ} we will use the fact that, for two quantities $a, b:\mathbb{N}^2_0\to \mathbb{R}$, 
there holds 
  \begin{align*}
    \sum_{\betam = \ell} \sum_{j=0}^{\ell-1} \sum_{\substack{\alpham=j\\\alpha\leq\beta}} a(\alpha)b(\beta-\alpha)
    &=
    \sum_{j=0}^{\ell-1}\sum_{\betam = \ell}  \sum_{\substack{\alpham=j\\\alpha\leq\beta}} a(\alpha)b(\beta-\alpha)
    =
    \sum_{j=0}^{\ell-1}\sum_{\alpham = j}  \sum_{\substack{\betam=\ell\\\beta\geq\alpha}} a(\alpha)b(\beta-\alpha)\\
    &=
    \sum_{j=0}^{\ell-1}\sum_{\alpham = j}  \sum_{\xim = \ell-j} a(\alpha)b(\xi) 
     \;.
  \end{align*}
With this summation identity, and due
to Lemma \ref{lemma:nonlin1}, Lemma \ref{lemma:nonlin2}, and Stirling's inequality,
there exists constants $C_i$, $i=1,2,3,4$ independent of
$\ell$, $k$, $\rho$ and $A_{\bu}$, such that
\begin{equation}
  \label{eq:product-inequ}
\begin{aligned}
&\sum_{\betam=\ell}\sum_{j=0}^{\ell-1} \sum_{\substack{\alpham=j\\\alpha\leq\beta}}\binom{\beta}{\alpha} \| r^{\frac{1-\gamma}{2}+\betam-\alpham} \partial^{\beta-\alpha} u_m \|_{L^{2 s} (S_{R-k\rho})} \| r^{\frac{1-\gamma}{2}+1+\alpham} \dalpha \partial_m u_n \|_{L^{2 s} (S_{R-k\rho})} \\
  &\qquad \leq \sum_{j=0}^{\ell-1} \binom{\ell}{j}\sum_{\alpham=j}\sum_{\xim=\ell-j} \| r^{\frac{1-\gamma}{2}+\xim}
    \dxi u_m \|_{L^{2 s} (S_{R-k\rho})} \| r^{\frac{1-\gamma}{2}+1+\alpham} \dalpha \partial_m u_n \|_{L^{2 s} (S_{R-k\rho})}\\
&\qquad\leq C_{1} A_{\bu}^{\ell +1+2\theta} \rho^{-\ell-1-2\theta} \sum_{j=0}^{\ell-1} \binom{\ell}{j} \frac{(\ell-j +1)^{\ell-j+\theta} (j+2)^{j +1+\theta}}{k^{\ell+1+2\theta}}\\
  &\qquad\leq
      C_{2} A_{\bu}^{\ell +1+2\theta} \rho^{-\ell-1-2\theta}
      \sum_{j=0}^{\ell-1} \binom{\ell}{j} (j+1)!(\ell - j)! e^\ell \frac{(j+2)^\theta (\ell-j+1)^\theta}{k^{\ell+1+2\theta} \sqrt{(j+2)(\ell-j+1)}}
    \\
&\qquad\leq C_{3} A_{\bu}^{\ell +1+2\theta} \rho^{-\ell-1-2\theta} \frac{\ell!e^\ell}{k^\ell}\sum_{j=0}^{\ell-1}\frac{(j+1)}{k}\frac{1}{\sqrt{(j+2)(\ell-j+1)}}\\
  &\qquad\leq C_{4} A_{\bu}^{\ell +1+2\theta} \rho^{-\ell-1-2\theta} \left(\frac{\ell}{k}\right)^{\ell+1}\ell^{1/2}.
 \end{aligned}
\end{equation}
  Consider now the remaining term in \eqref{eq:ugu-proof1}. This term has to be
  treated differently due to the possibility that $\ell = k-1$. There holds, using
  \eqref{eq:thetaprod-infty}, \eqref{eq:induction-u}, and writing $\eta = 2/s$,
  \begin{align*}
  & \sum_{\betam=\ell}\| r^{2-\gamma +\betam }u_m \dbeta \partial_m u_n\|_{L^s(S_{R-k\rho})}\\
    & \qquad \qquad \leq 
 \sum_{\betam=\ell} \|r u_m \|_{L^\infty(S_{R-k\rho})}\| r^{1-\gamma +\betam } \dbeta \partial_m u_n\|_{L^s(S_{R-k\rho})}\\
    & \qquad \qquad \leq
C_{5}\|r^{-\gamma}  u_m\|^{1-\eta}_{L^s(S_{R-k\rho})}
  \| u_m\|_{\cK^{1,s}_\gamma(S_{R-k\rho})}^{\eta}
\sum_{\betam=\ell}\| r^{1-\gamma +\betam } \dbeta \partial_m u_n\|_{L^s(S_{R-k\rho})}\\
    & \qquad \qquad \leq
C_{6} A_{\bu}^{\ell +1+\eta} (k\rho)^{-\ell-1-\eta} (\ell+1)^{\ell+1}.
  \end{align*}
  with $C_5$, $C_6$ independent of $A_{\bu}$, $\ell$, $k$, and $\rho$. This concludes the proof.
  \end{proof}
\section{The Stokes and Navier-Stokes equations in a polygon $P$}
\label{sec:polygon}
We now consider the Navier-Stokes problem in a polygon $P$, 
with internal angles $\phi_i$ such that $0<\phi_i<2\pi$, 
corners $c_i$ and edges $e_i$, $i=1,\dots, n$ 
as depicted in Figure \ref{fig:polygon}. 
\begin{figure}
  \centering
  \includegraphics[width=.5\textwidth]{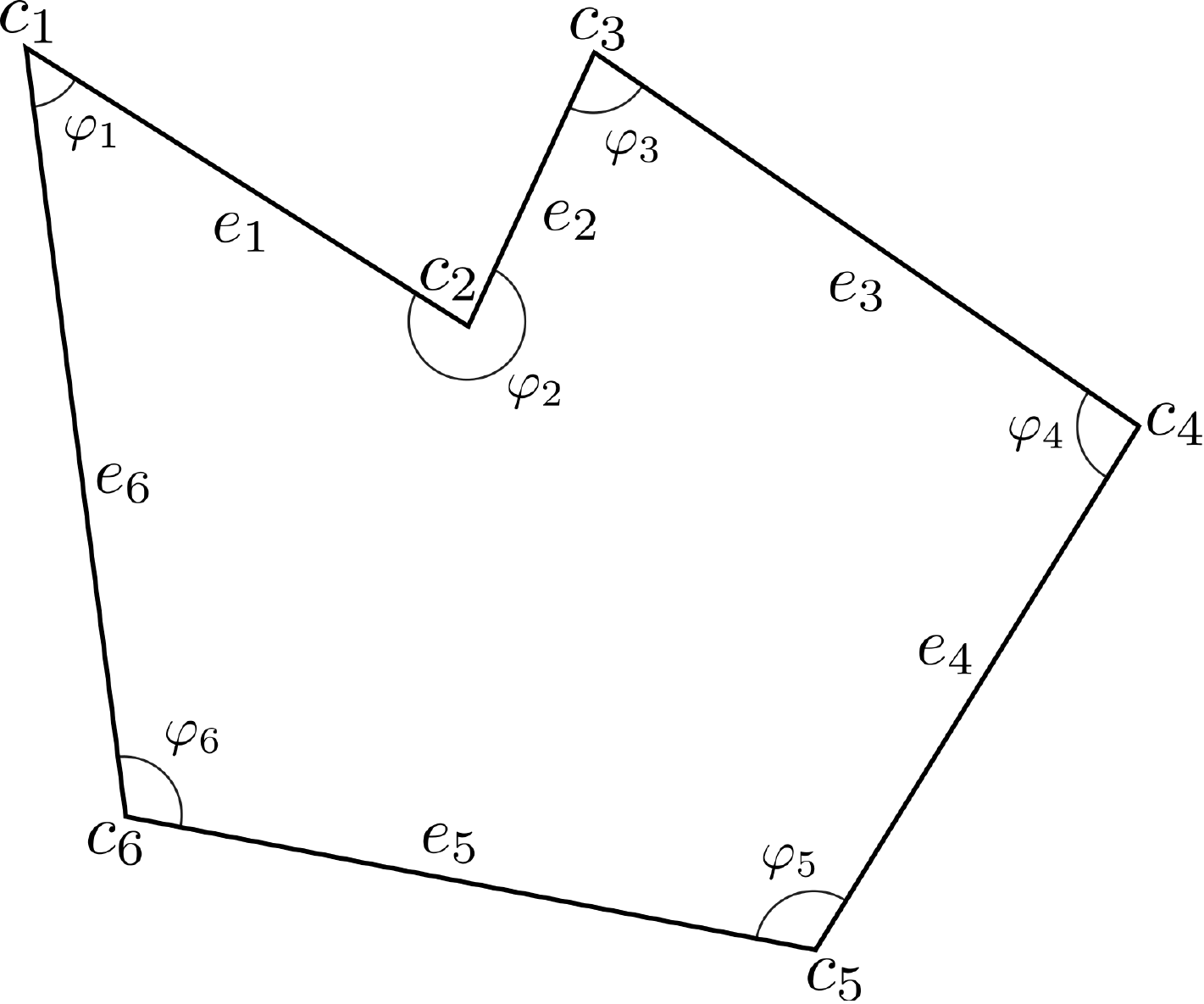}
  \caption{The polygon $P$, with internal angles $\phi_i$, corners $c_i$, and edges $e_i$, $i=1,
    \dots, 6$.}
  \label{fig:polygon}
\end{figure}
We recall \eqref{eq:NSEIntr},
the incompressible Navier-Stokes equations with homogeneous Dirichlet boundary conditions
\begin{equation}
  \label{eq:ns-P}
\begin{split}
  -\nu \Delta \bu +\left( \bu\cdot\nabla \right)\bu+ \nabla p  &= \bbf \text{ in }P,\\
  \nabla \cdot \bu &= 0  \text{ in }P,\\
  \bu &= \bzero \text{ on }\partial P.
\end{split}
\end{equation}
It is well-known that for sufficiently small data, 
i.e.\ for sufficiently small $\| f \|_{\bm{H}^{-1}(P)}/\nu^{2}$,
\eqref{eq:ns-P} admits a unique weak (Leray-Hopf) solution 
$(\bu,p) \in H^1_0(P)^2 \times L^2(P)/{\mathbb R}$ 
(e.g. \cite{Girault1986a,Ladyzhenskaya1969,LionsJL69,Temam1995}).
For convenience, \emph{we assume throughout what follows that 
$\nu = 1$ and that $\| \bbf \|_{\bm{H}^{-1}(P)}$ is sufficiently small 
to ensure uniqueness of $(\bu,p)$}.
The proof of weighted analytic estimates for Leray-Hopf weak solutions of 
problem \eqref{eq:ns-P} consists in the combination of the estimates 
localised at every corner sector with classical regularity results 
in the interior and along the (straight) edges of the domain.

To show the estimates in the corner sectors, we will again 
proceed by induction.
Specifically, we decouple the inductive step and the
starting point of our argument.
First, in Lemma \ref{lemma:induction-step}, 
we prove an analytic regularity shift estimate, starting from
the induction hypothesis introduced in \eqref{eq:induction} and using the results
obtained in Section \ref{sec:nonlin-term}. 
This will serve in the following as the local induction
step for the proof of analytic regularity in polygons.
Then, the base case for induction is given
in Lemma \ref{lemma:base-reg}, where we derive a basic regularity estimate in
weighted spaces for the Leray-Hopf solution to the problem under consideration. 

Theorem \ref{th:analytic-P} constitutes then the main result of this paper, i.e.,
weighted analytic regularity of the Leray-Hopf solutions to the incompressible Navier-Stokes problem, 
in polygonal domains and with ``no-slip'', homogeneous Dirichlet boundary conditions.

\subsection{The Navier-Stokes equation in the polygon}
We start by the introduction of the setting and notation for the polygonal case.
For all $i=1, \dots, n$ and for $0< D <|e_i|$, where $|e_i|$ is the length
of the edge $e_i$, we denote the sectors centered at corner $c_i$ by
\begin{equation*}
  S^i_D = B_D(c_i)\cap P
\end{equation*}
where $B_D(c_i)$ is the ball of radius $D$ centered in $c_i$.
We write $r_i = |x-c_i|$, $\ugamma = \{\gamma_i\}_{i=1,\dots, n}$ and
\begin{equation*}
  r^\ugamma = \prod_{i=1, \dots, n} r_i^{\gamma_i}.
\end{equation*}
For $\ell\in \mathbb{N}_0$, $1<s<\infty$ and 
for $\ugamma \in \mathbb{R}^n$, 
the definition of the spaces $\cK^{\ell, s}_\ugamma(P)$ and $\cK^{\varpi,s}_\ugamma(P)$ 
follows directly from the one introduced for $K$.
In the following, comparisons between $n$-dimensional vectors have to be
  interpreted elementwise, e.g., for $\ugamma = \{\gamma_i\}_{i=1, \dots, n}$
  and $ \ueta = \{\eta_i\}_{i=1, \dots, n}$, we write
  $\ugamma < \ueta$ if 
  \begin{equation*}
    \gamma_i  < \eta_i \quad \text{for all}\quad i=1, \dots, n. 
  \end{equation*}
  Also, for $x\in \mathbb{R}$, we write $x+\ugamma = \{x+\gamma_i\}_{i=1, \dots, n}$.

Fix $R_P\in \mathbb{R}_+$ such that $R_P\leq 1$ and $R_P \leq \min_{i,j=1,\dots, n} |c_i-c_j|$. 
Denote by $P_0$ a set bounded away from the corners such that $P_0 \cup \left(
  \bigcup_{i=1, \dots, n} S^i_{R_P/2}\right) = P$, see Figure
  \ref{fig:polygon-decomposition}.
  \begin{figure}
    \centering
    \includegraphics[width=.4\textwidth]{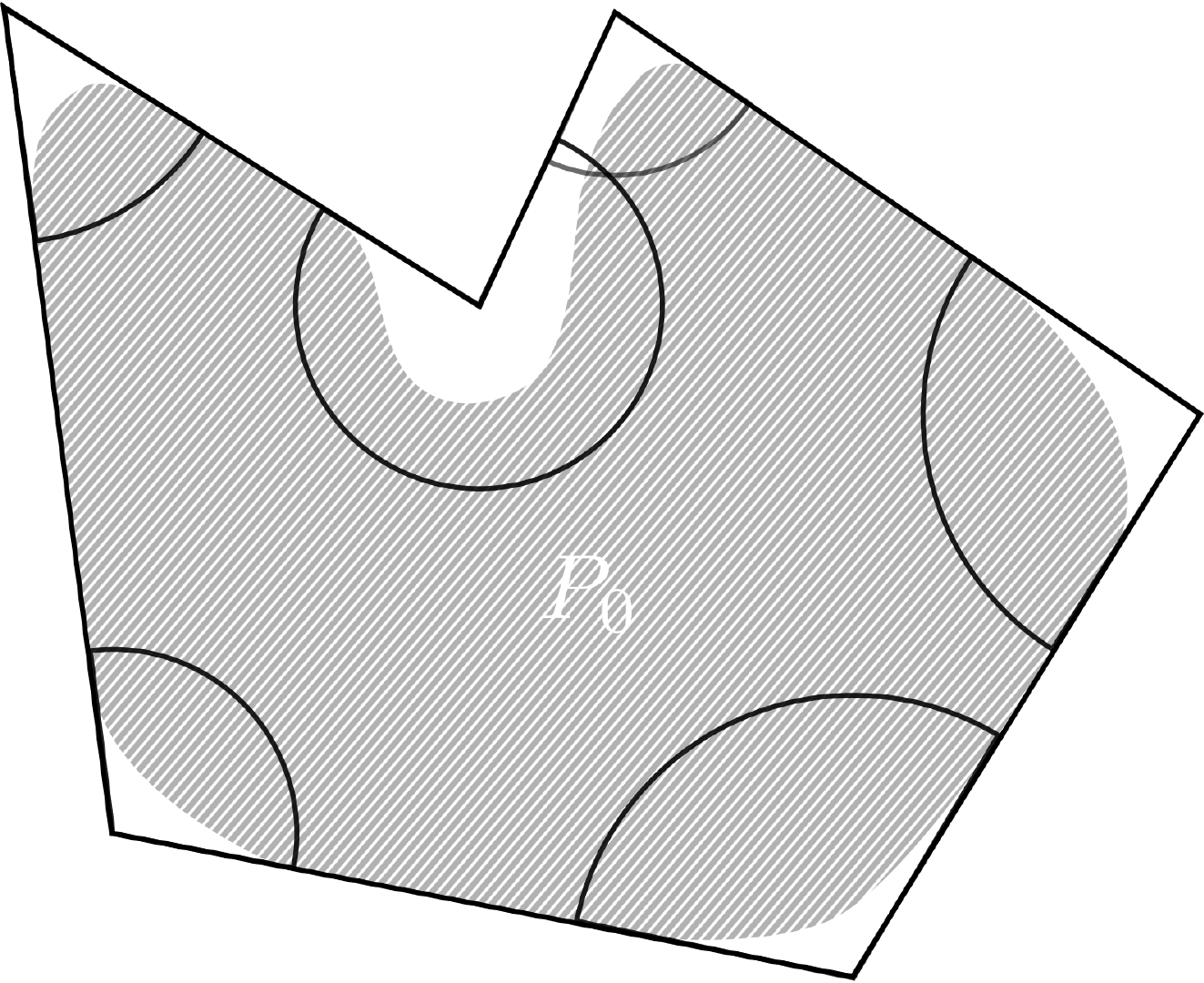}
    \caption{Corner sectors with different radii and set $P_0$, denoted by the shaded part of the polygon.}
    \label{fig:polygon-decomposition}
  \end{figure}

\subsubsection{Inductive step}
  Using the results introduced in Sections \ref{sec:sector}, we can now show that, given a sufficiently
  regular right hand side, if the derivatives of a Leray-Hopf solution to the Navier-Stokes
  equations satisfy weighted analytic estimates up to a certain order in a sector, 
  then by elliptic regularity in weighted spaces the estimates hold to higher order. 
  This will be used as the
  inductive step in the proof of weighted analytic regularity in polygons.
  \begin{lemma}
    \label{lemma:induction-step}
 Let $k\in \mathbb{N}$, $i\in \{1, \dots, n\}$, $2
 < s<\infty$, $\gamma \in \mathbb{R}$ such that
  $\gamma-2/s>-1$. Let 
  \begin{equation*}
    \bbf\in \cK^{\varpi, s}_{\gamma-2}(S^i_{R_P}; C_f, A_f),
  \end{equation*}
  and let  
  $(\bu, p)\in\bcK^{k}_\gamma(S^i_{R_P-k\rho})\times \cK^{k-1}_{\gamma-1}(S^i_{R_P-k\rho})$ 
  be a solution to \eqref{eq:ns-P}. 
  Then, $C_{\bu}$, $A_{\bu}$, $C_p$, and $A_p$ can be chosen independently of $k$ and $\rho$ 
  so that the hypothesis ``$H_{\bu,p}(s,\gamma, k,C_{\bu,p}, A_{\bu,p})$ holds in $S^i_{R_P}$''
  implies
  \begin{equation*}
    H_{\bu,p}(s,\gamma, k+1,C_{\bu,p}, A_{\bu,p})\text{ holds in }S^i_{R_P}\;.
  \end{equation*}
\end{lemma}
\begin{proof}
  We remind that we work under the small data hypothesis which ensures that 
      $(\bu,p)$ is unique.  

Assume that $H_{\bu,p}(s,\gamma, k,C_{\bu,p}, A_{\bu,p})$ holds in
  $S^i_{R_P}$. Then, for all $\rho\in(0, R_P/(2(k+1))]$, there exists
  $\hrho = (k+1)\rho/k$ such that
  \begin{equation*}
    |\bu |_{\bcK^{\ell,s}_\gamma(S_{R_P-(k+1)\rho}) }= 
    |\bu |_{\bcK^{\ell,s}_\gamma(S_{R_P-k\hrho}) }\leq 
    C_{\bu} A_{\bu}^{\ell}(k\hrho)^{-\ell}\ell^\ell =
    C_{\bu} A_{\bu}^{\ell}((k+1)\rho)^{-\ell}\ell^\ell 
  \end{equation*}
  for all $\ell = 0, \dots, k$. Similarly,
\begin{equation*}
    |p |_{\cK^{\ell-1,s}_{\gamma-1}(S_{R_P-(k+1)\rho}) }= 
    |p |_{\cK^{\ell-1,s}_{\gamma-1}(S_{R_P-k\hrho}) }\leq 
    C_{p} A_{p}^{\ell}(k\hrho)^{-\ell}\ell^\ell =
    C_{p} A_{p}^{\ell}((k+1)\rho)^{-\ell}\ell^\ell 
  \end{equation*}
  for all $\ell=1, \dots, k$. Therefore, it remains to prove the bounds on
    $|\bu |_{\bcK^{k+1,s}_\gamma(S_{R_P-(k+1)\rho}) } $ and on $|p |_{\cK^{k,s}_{\gamma-1}(S_{R_P-(k+1)\rho}) }$.
      
      We write \eqref{eq:ns-P} as (recall that we assume $\nu=1$)
\begin{equation*}
\begin{split}
  - \Delta \bu + \nabla p  &= \bbf - \left( \bu\cdot\nabla \right)\bu\;\text{ in }P,\\
  \nabla \cdot \bu &= 0  \text{ in }P,\\
  \bu &= \bzero \text{ on }\partial P\setminus\{0\}.
\end{split}
\end{equation*}
By the hypotheses and Lemma \ref{lemma:nonlin-est}, 
$\bbf - \left( \bu\cdot\nabla \right)\bu\in \bcK^{k-1, s}_{\gamma-2}(S^i_{R_P-k\rho})$. 
We apply Proposition \ref{th:stokes-estimate-bis} to obtain,
for all $\rho \in (0, R_P/(2(k+{1}))]$,
\begin{align*}
&| \bu |_{\bcK^{k+1,s}_\gamma (S^i_{R_P-(k+1)\rho})}
  +     | p |_{\cK^{k,s}_{\gamma-1}(S^i_{R_P-(k+1)\rho})} 
  \\
  &\quad
  \leq 
    C_{1}
    \begin{multlined}[t][.8\textwidth]
    \bigg\{
    |  \bbf |_{\bcK^{k-1,s}_{\gamma-2}(S^i_{R_P-k\rho})}
    + |  (\bu \cdot \nabla)\bu |_{\bcK^{k-1,s}_{\gamma-2}(S^i_{R_P-k\rho})}
  \\
    +\rho^{-1}  |  p |_{\cK^{k-1,s}_{\gamma-1}(S^i_{R_P-k\rho})}
  + \sum_{n \leq 1}\rho^{n-2}|  \bu |_{\bcK^{k+n-1,s}_{\gamma}(S^i_{R_P-k\rho})}
  \bigg\},
    \end{multlined}
\end{align*}
with $C_1$ independent of $k$, $\rho$, $\bu$, and $p$.
Now, if $H_{\bu,p}(s,\gamma, k,C_{\bu,p}, A_{\bu,p})$ holds in $S^i_{R_P}$
and using the estimate \eqref{eq:nonlin-est} of Lemma \ref{lemma:nonlin-est}, 
we obtain
\begin{equation}
\label{eq:pre-end-inductive-step}
\begin{aligned}
&  | \bu |_{\bcK^{k+1,s}_\gamma (S^i_{R_P-(k+1)\rho})}
  +     | p |_{\cK^{k,s}_{\gamma-1}(S^i_{R_P-(k+1)\rho})} \\
  & \qquad \leq
  C_{2}
  \bigg\{
     A_f^{k-1} (k-1)! +  A_{\bu}^{k+2/s}(k\rho)^{-k-2/s} k^{k+1/2}
      + A_{p}^{k}\rho^{-k-1}
      + A_{\bu}^{k}\rho^{-k-1}
    \bigg\},
\end{aligned}
\end{equation}
with $C_2\geq 1$ independent of $k, \rho$, $A_{\bu}$, and $A_p$.
Thus, choosing $C_{\bu}, C_p \geq 1$ and $A_{\bu} = A_p\geq 1$ 
so large that
\begin{equation} \label{eq:end-inductive-step}
    A_{\bu}\geq A_f, \qquad  A_{\bu}^{1-2/s} \geq 4C_{2},
\end{equation}
and that $H_{\bu,p}(s,\gamma, k,C_{\bu,p}, A_{\bu,p})$ holds in $S^i_{R_P}$,
we obtain the assertion.
\end{proof}
\begin{remark}
  \label{remark:s>2}
 The proof of Lemma \ref{lemma:induction-step} shows how the
 argument used in the present paper does not work when $s\leq 2$.
 Specifically, only if $s>2$ there exists an $A>0$ such that the second inequality
 in \eqref{eq:end-inductive-step} is verified and $(k\rho)^{-2/s}<
   (k\rho)^{-1}$ in \eqref{eq:pre-end-inductive-step}.
\end{remark}

\subsubsection{Base regularity estimate}

We recall two lemmas characterizing continuous imbeddings in weighted
spaces from \cite[Lemmas 4.1.2 and 4.1.3]{Mazya2010}.
\begin{lemma}
  \label{lemma:imbedding-sameorder}
  Let $1<t < s<\infty$, $\ugamma, \ueta\in \mathbb{R}^n$, and $\ell\in\mathbb{N}_0$ 
  such that $\ugamma -2/s > \ueta -2/t $.
  Then,
  \begin{equation*}
    \cK^{\ell,s}_{\ugamma}(P) \hookrightarrow \cK^{\ell,t}_{\ueta}(P).
  \end{equation*}
\end{lemma}

\begin{lemma}
  \label{lemma:imbedding}
  Let $1<s\leq t<\infty$, $\ugamma, \ueta\in \mathbb{R}^n$, and $\ell, k$ in $\mathbb{N}$ such that $\ell > k$ and $\ell-2/s \geq k-2/t$, $\ugamma -2/s \geq \ueta -2/t $.
  Then,
  \begin{equation*}
    \cK^{\ell,s}_{\ugamma}(P) \hookrightarrow \cK^{k,t}_{\ueta}(P).
  \end{equation*}
\end{lemma}

We 
next introduce a regularity result in weighted spaces for the Navier-Stokes problem \eqref{eq:ns-P}. 
This will constitute the base case for the inductive proof of analytic regularity. 
In the following we investigate the regularity in the weighted spaces $\cK^{1, s}_\gamma(P)$ 
for $s>2$, i.e., in the non Hilbertian setting. 
Our strategy to prove the base case is to 
prove a higher order regularity result, then use the preceding imbeddings to prove
the base case.
\begin{lemma}
  \label{lemma:base-reg}
  Let $\bu$, $p$ be a solution to \eqref{eq:ns-P}. 
  Suppose that there exists $2\leq \ts$ and $\utgamma\in \mathbb{R}^n$ such that
  \begin{equation*}
    \bbf \in \bm{H}^{-1}(P) \cap \bcK^{0,\ts}_{\utgamma-2}(P)
    .
  \end{equation*}
  Then,
  for all $1<s < \infty$ and $\ugamma \in \mathbb{R}^n$ such that
  $\ugamma -2/s \leq \min(2, \utgamma+1-2/\ts)$, there holds
  \begin{equation*}
  \bu \in \bcK^{1,s}_{\ugamma}(P), \qquad p \in \cK^{0,s}_{\ugamma-1}(P).
  \end{equation*}
  \end{lemma}
\begin{proof}
Since $\bu$ is solution to \eqref{eq:ns-P}, $\bu\in \bm{H}^1_0(P) = \bcK^{1,2}_{\underline{1}} (P)$, 
where $\underline{1} = \{1, \dots, 1\} \in \mathbb{R}^n$, see \cite{Girault1986a}. 
Hence, by the H\"{o}lder inequality and by Sobolev imbedding,
  for all $1< s_1< 2$ and $\utgamma_1 - 2/s_1 < -1$, denoting $t = (1/s_1- 1/2)^{-1}$,
  \begin{equation*}
    \| r^{-\utgamma_1} (\bu\cdot \nabla)\bu \|_{L^{s_1}(P)} \leq \| r^{-\utgamma_1} \bu\|_{L^{t}(P)} \|\nabla \bu \|_{L^2(P)} < \infty,
  \end{equation*}
  i.e.,
  \begin{equation*}
    (\bu\cdot \nabla)\bu\in \bcK^{0,s_1}_{\utgamma_1}(P).
  \end{equation*}
 We temporarily fix $s_1$ and $\utgamma_1$. 
By Lemma \ref{lemma:imbedding-sameorder},
\begin{equation*}
\bbf \in  \bcK^{0,s_1}_{\ueta}(P) \quad \text{for all}\quad\ueta <\utgamma-2+2/s_1-2/\ts. 
\end{equation*}
Problem \eqref{eq:ns-P} can be rewritten as
\begin{equation*}
\begin{split}
  -\Delta \bu + \nabla p  &= \bbf -\left( \bu\cdot\nabla \right)\bu\text{ in }P,\\
  \nabla \cdot \bu &= 0  \text{ in }P,\\
  \bu &= \bzero \text{ on }\partial P.
\end{split}
\end{equation*}
  Then, by classical elliptic regularity in $P_0$ -- see for example
  \cite{Girault1986a} -- and using Theorem \ref{th:stokes-homogeneous-reg} in all the
  sectors $S^i_{R_P}$, $i=1, \dots, n$, we obtain
  \begin{equation*}
    \bu\in \bcK^{2, s_1}_{\ugamma_1}(P), \qquad p \in \cK^{1,s_1}_{\ugamma_1-1}(P),
  \end{equation*}
  for all $\ugamma_1\in \mathbb{R}^n$ such that 
  $\ugamma_1 - 2/s_1 \leq \utgamma_1+2$ and $\ugamma_1 - 2/s_1< \utgamma +2/s_1 - 2/\ts$. 
  Then, by continuous imbedding -- see Lemma \ref{lemma:imbedding} -- 
  if $s_1\leq s_2 \leq 2s_1/(2-s_1)$ and $\ugamma_2 - 2/s_2 \leq  \ugamma_1 - 2/s_1$, 
  there holds
  \begin{equation}
    \label{eq:base-proof-imbed}
    \bu\in \bcK^{1, s_2}_{\ugamma_2}(P), \qquad p \in \cK^{0,s_2}_{\ugamma_2-1}(P).
  \end{equation}
  Due to the arbitrariness of $ 1 <  s_1 < 2$, 
  \eqref{eq:base-proof-imbed} holds for all 
  $ 2 \leq s_2 < \infty$, $\ugamma_2 - 2/s_2  \leq \min(2, \utgamma+1-2/\ts)$.
\end{proof}
\subsubsection{Analytic regularity}
\label{sec:AnReg}
We denote by $\mathfrak{A}^i(\lambda)$ the Mellin symbol of the Stokes operator associated to
the corner $c_i$, i.e., the operator $\mathfrak{L}(\theta, \partial_\theta; \lambda)$ 
in $(0, \phi_i)$ with boundary operator 
$\mathfrak{B}(\theta, \partial_\theta; \lambda)_{|_{\theta= 0, \phi_i}}$ 
introduced in Section \ref{sec:stokes}. 
The operator pencil $\mathfrak{A}^i(\lambda)$ is Fredholm and therefore 
    admits a discrete spectrum of countably many, possibly complex eigenvalues
    which accumulate only at $\lambda = +\infty$.
Let $\Lambda^i_{\min{}}$ be the real part of the eigenvalue with smallest
positive real part of $\mathfrak{A}^i(\lambda)$, and 
let $\uLambda~=~\{ \Lambda^i_{\min{}}\}_{i=1, \dots, n} >0$.

The regularity in countably normed, weighted analytic spaces
for the solution to the incompressible Navier-Stokes equation in the polygon $P$ 
then follows by a bootstrapping argument with induction with respect to the order
of differentiation in the Kondrat'ev scales, with parameter $\gamma$
in one common strip of width $2\uLambda$ 
determined by the spectra of  $\mathfrak{A}^i(\lambda)$, $i=1,...,n$
as in finite-order regularity results in \cite{Grisvard1985,OrltSndig95}.
\begin{theorem}
  \label{th:analytic-P}
  Let $2 < s<\infty$, $\ugamma \in \mathbb{R}^n$ such that
  $\max(-\uLambda , -1)< \ugamma-2/s < \min(\uLambda, 1)$. Let
  $(\bu, p)$ be solution to \eqref{eq:ns-P} with 
  \begin{equation*}
    \bbf\in \bcK^{\varpi, s}_{\ugamma-2}(P)\cap \bm{H}^{-1}(P)
    .
  \end{equation*}
  Then, 
  \begin{equation*}
    \bu\in \bcK^{\varpi, s}_{\ugamma}(P), \qquad
    p\in \cK^{\varpi, s}_{\ugamma-1}(P).
  \end{equation*}
\end{theorem}
\begin{proof}
We rewrite for the polygon $P$ the induction hypotheses introduced in Section \ref{sec:nonlin-term}. 
For $1 <s <\infty$, $\gamma \in \mathbb{R}^n$, 
$i=1, \dots, n$, $C_{\bu}, A_{\bu}, C_p,A_p>0$, 
and  for $j\in \mathbb{N}_0$, we introduce two statements:
  we shall say that $H^i_{\bu}(s, \ugamma, j,C_{\bu}, A_{\bu})$ holds 
  if for all $\rho\in (0, R_P/(2j)]$,
  \begin{subequations}
  \begin{equation}
    \label{eq:induction-P-u}
    \left| \bu \right|_{\bcK^{\ell,s}_{\gamma_i}(S^i_{R_P-j\rho})} 
    \leq C_{\bu} A_{\bu}^\ell (j\rho)^{-\ell} \ell^\ell \text{ for all }\ell =0, \dots, j
  \end{equation}
  and $H^i_p(s,\ugamma,j,C_{p}, A_{p})$ holds if for all $\rho\in (0, R_P/(2j)]$,
  \begin{equation}
    \label{eq:induction-P-p}
    \left| p \right|_{\cK^{\ell-1,s}_{\gamma_i-1}(S^i_{R_P-j\rho})} 
       \leq C_{p} A_{p}^\ell (j\rho)^{-\ell} \ell^\ell \text{ for all }\ell=1, \dots, j.
  \end{equation} 
\end{subequations}
For integer $k$, the induction hypothesis  is formulated as 
\begin{equation}
  \label{eq:induction-P}
  \begin{multlined}
    H_{\bu,p}(s,\ugamma,  k,C_{\bu,p}, A_{\bu,p}):
    \\
  \qquad\left\{\begin{aligned}
   &H^i_{\bu}(s,\ugamma, j,C_{\bu}, A_{\bu}) \text{ holds for }j=0, \dots, k, 
   \\
  &H^i_p(s,\ugamma, j,C_{p}, A_{p})\text{ holds for }j=1, \dots, k, 
  \end{aligned}
  \right.
  \text{ for all }i = 1, \dots, n.
  \end{multlined}
\end{equation}
  Fix $2< s< \infty$ and $\ugamma$ such that the hypotheses of the theorem
  hold. Then, by
  Lemma \ref{lemma:base-reg}, $\bu \in \bcK^{1,s}_{\ugamma}(P)$ and $p\in
  \cK^{0,s}_{\ugamma-1}(P)$, i.e., there exist $C_{\bu,p}, A_{\bu,p}$ such
    that $H_{\bu,p}(s, \ugamma, 1, C_{\bu,p}, A_{\bu,p})$ holds.
  Suppose now that, for a $k\in \mathbb{N}$, $H_{\bu,p}(s, \ugamma, k, C_{\bu,p}, A_{\bu,p})$ holds,
  with $C_{\bu,p}, A_{\bu,p}$ chosen so large that the conclusion of Lemma
  \ref{lemma:induction-step} holds in every corner sector.
  Then, by Lemma \ref{lemma:induction-step} applied to every corner sector,
  $H_{\bu,p}(s, \ugamma, k+1 ,C_{\bu,p}, A_{\bu,p})$ holds. 
  By induction, this implies that, for all $i=1, \dots, n$,
  \begin{equation}
   \label{eq:an-Si}
    \bu\in \cK^{\varpi, s}_{\ugamma}(S^i_{R_P/2}), \qquad
    p\in \cK^{\varpi, s}_{\ugamma-1}(S^i_{R_P/2}).
  \end{equation}
  The set $P_0$ can be covered by a finite number of balls, and from classical
  results of interior analyticity and of analyticity in regular parts of the
  boundary, see, e.g., \cite{Masuda1967a, Giga1983}, we derive analyticity in
  $P_0$. Since $r$ is bounded away from zero in $P_0$, this implies
 \begin{equation}
   \label{eq:an-P0}
    \bu\in \cK^{\varpi, s}_{\ugamma}(P_0), \qquad
    p\in \cK^{\varpi, s}_{\ugamma-1}(P_0).
  \end{equation}
  Combining \eqref{eq:an-Si} for $i=1, \dots, n$ and \eqref{eq:an-P0} implies the assertion.
\end{proof}
\subsection{Approximability of the solutions}
\label{sec:ApprSol}
We briefly outline some consequences of the weighted analytic
regularity result on Kolmogorov $N$-widths of solution families. 
As is well-known, bounds on Kolmogorov $N$-widths of solution families
of \eqref{eq:NSEIntr} will imply convergence rate estimates for 
for model order reduction and reduced basis methods for NSE.
See, e.g.,  \cite{Lassila2014,Hesthaven2016} and the references there
for examples and details.

We start with a corollary of Theorem \ref{th:analytic-P},
based on an embedding in weighted spaces.
\begin{corollary}
\label{corollary:analytic-L2}
Let $2 < s<\infty$, $\ugamma_{f} \in \mathbb{R}^n$ such that
$\ugamma_{f}-2/s > 0$. 
Let
$(\bu, p)$ be solution to \eqref{eq:ns-P} with 
$\bbf\in \bcK^{\varpi, s}_{\ugamma_{f}-2}(P)\cap \bm{H}^{-1}(P)$.
Then there exists a weight vector $\ugamma > 1$ such that
\begin{equation*}
    \bu\in \bcK^{\varpi,2}_{\ugamma}(P), \qquad
    p\in \cK^{\varpi,2}_{\ugamma-1}(P).
  \end{equation*}
\end{corollary}
\begin{proof}
From Theorem \ref{th:analytic-P}, 
$\bu\in \bcK^{\varpi,s}_{\ugamma_f}(P)$ 
and 
$p\in \cK^{\varpi,s}_{\ugamma_f-1}(P)$.
Then, let $1/q = 1/2 - 1/s$. 
By Hölder's inequality,
\begin{equation*}
\| r^{\alpham - \ugamma} \dalpha \bu \|_{L^2(P)} 
\leq 
C \| r^{\alpham - \ugamma_f} \dalpha \bu \|_{L^s(P)} \| r^{\ugamma_f-\ugamma}  \|_{L^q(P)} .
\end{equation*}
Fix any $\ugamma\in\mathbb{R}^n$ such that $\ugamma < \ugamma_f -2/s +1$. 
The second norm on the right hand side
of the above inequality is then bounded independently of $\alpham$. 
This implies $\bu\in \bcK^{\varpi,2}_{\ugamma}(P)$. 
Now, since $\ugamma_f - 2/s > 0$, 
we obtain the first part of the assertion.
The regularity of
the pressure follows by the same embedding argument.
\end{proof}
From Theorem \ref{th:analytic-P},
we can also derive an upper bound on the approximability of the set of solutions
to \eqref{eq:ns-P}. 
To this end, given a Hilbert space $X$, a compact subset $S$ of $X$, 
and $N\in \mathbb{N}$,
we recall the notion of \emph{Kolmogorov $N$-width of $S$ in $X$} 
which is defined as 
\begin{equation} \label{eq:nwidth}
d_N(S, X) =\inf_{\substack{X_N\subset X \\ \dim(X_N) = N}} 
\sup_{x\in S} \inf_{x_N\in X_N} \|x-x_N\|_X.
\end{equation}
In \eqref{eq:nwidth}, 
$X_N \subset X$ denotes any linear subspace of $X$ of finite dimension $N$.

The following Lemma bounds the $N$-widths of the 
weighted analytic function spaces.
  \begin{lemma}
   \label{lemma:nwidth-1} 
    Let $C_{\bu}, A_{\bu}, C_p,A_p>0$ and $\ugamma >  1$. 
    Then, there exist $C, b >0$ such that
  \begin{equation*}
  d_N(\cK^{\varpi,2}_\ugamma(P; C_{\bu}, A_{\bu}), H^1(P))  \leq C\exp\left(-b N^{1/3}\right)
  \end{equation*}
   and
  \begin{equation*}
  d_N(\cK^{\varpi,2}_{\ugamma-1}(P; C_p, A_p), L^2(P))  \leq C\exp\left(-b N^{1/3}\right).
  \end{equation*}
  \end{lemma}
  \begin{proof}
    We recall the definition of the countably normed spaces $B^\ell_\beta(P)$ as, e.g., in
    \cite{Guo2006a}. For $\ell\in \mathbb{N}$, $\ubeta \in \mathbb{R}^n$, and
    $c,d>0$, a
    function $v$ belongs to $B^\ell_\ubeta(P; c, d)$ if $v\in H^{\ell-1}(P)$ 
    and, for all $\alpham \geq \ell$
    \begin{equation*}
     \| r^{\alpham + \ubeta - \ell}  \dalpha v \|_{L^2(P)} \leq c d^{\alpham - \ell}(\alpham - \ell)!.
    \end{equation*}
    The spaces $B^\ell_\beta(P)$ correspond to non-homogeneous versions of the
    spaces $\cK^\varpi_\gamma(P)$, see \cite{Costabel2010a}, 
    for a thorough analysis of their relationship.

    Let $\ugamma>1$: we show that for all $\tA>A$, there exists $\tC$ such that
    $\cK^\varpi_{\ugamma}(P; C, A) \subset
    B^2_{2-\ugamma}(P; \tC, \tA)$. If $v\in\cK^\varpi_\ugamma(P; C, A)$, one clearly has
    \begin{equation} \label{eq:BK1}
    \| v \|_{H^1(P)} \leq 
    \max(\|r^\ugamma\|_{L^\infty(P)},  \|r^{\ugamma-1}\|_{L^\infty(P)})\|v\|_{\cK^1_\ugamma(P)},
    \end{equation}
    hence $v\in H^1(P)$. 
    Furthermore, by definition, there exist $C, A>0$ such that
    \begin{equation} \label{eq:BK2}
     \| r^{\alpham + (2-\ugamma)-2}  \dalpha v \|_{L^2(P)} 
     \leq 
     C A^{\alpham}\alpham! = CA^2 \alpham (\alpham-1) A^{\alpham-2}(\alpham-2)!.
    \end{equation}
    For any $\tA>A$ one can choose $\tC \geq \max_{k\geq 2}
    A^k\tA^{2-k}k(k-1)$ and from \eqref{eq:BK1} and \eqref{eq:BK2},  $v\in
    B^2_{2-\ugamma}(P; \tC, \tA)$.
    The inclusion $\cK^\varpi_{\ugamma-1}(P; C, A)\subset B^1_{2-\ugamma}(P;
    \tC, \tA)$ can be
    shown in the same fashion.

    The application of \cite[Theorem 3.1]{CSMSuri99}, resp. of \cite{SchoetWihlNM2003},
    which contain the construction of an exponentially convergent, 
    $H^1_0(P)^2\times L^2_0(P)$-conforming approximation of functions in 
    $B^2_{\ubeta}(P;\tC, \tA)^2\times B^1_{\ubeta}(P; \tC, \tA)$ for arbitrary
    $\ubeta\in (0,1)$, concludes the proof.
  \end{proof}
    We denote by $\cS$ the solution operator of \eqref{eq:ns-P}, i.e., we write
  $(\bu, p) = \cS(\bbf)$ if $(\bu, p)$ are the solutions to \eqref{eq:ns-P} with
  right hand side $\bbf$.

As a direct consequence of Corollary \ref{corollary:analytic-L2} and Lemma
  \ref{lemma:nwidth-1}, a second corollary to Theorem
  \ref{th:analytic-P} gives bounds on the
  $N$-widths of the sets of solutions to \eqref{eq:ns-P}.
  \begin{corollary}
    \label{corollary:nwidth}
    Let $2 < s<\infty$, $\ugamma_{f} \in \mathbb{R}^n$ such that
  $\ugamma_{f}-2/s > 0$, and $C_f, A_f>0$. 
  Let $(\bcU, \cP)$ be the sets of, respectively, 
  all velocities and pressures, solutions 
  of the Navier-Stokes equations \eqref{eq:ns-P} with right hand sides in 
  $\bcK^{\varpi,s}_{\ugamma_{f}-2}(P; C_f, A_f)\cap \bm{H}^{-1}(P)$, 
  i.e.,
  \begin{equation*}
    (\bcU, \cP) = \cS\left(\bcK^{\varpi, s}_{\ugamma_{f}-2}(P; C_f, A_f)\cap \bm{H}^{-1}(P)\right)
    \subset H^1_0(P)^2\times L^2_0(P).
  \end{equation*}
  Then, 
  $(\bcU, \cP)$ is a compact subset of  $H^1_0(P)^2\times L^2_0(P)$.
  Moreover, 
  there exist $C, b>0$  such that for all $N\in\mathbb{N}$ 
  \begin{equation*}
  d_N(\bcU \times \cP,  
  \left( H^1_0(P) \right)^2\times L_0^2(P))\leq C\exp\left(-bN^{1/3}\right).
  \end{equation*}
  \end{corollary}
  \begin{remark}
   By embedding in weighted spaces and rescaling of the constants, Corollary \ref{corollary:nwidth} 
   holds under the more usual condition that right hand sides of \eqref{eq:ns-P}
   belong to $\bcK^{\varpi,2}_{\ugamma_f-2}(P, C_f, A_f)$, with $\ugamma_f > 1$.
  \end{remark}

\section{Conclusion}
\label{sec:Concl}
We established analytic regularity of stationary solutions of the incompressible Navier-Stokes
equation in plane, polygonal domains, subject to the ``no-slip'', homogeneous Dirichlet
boundary condition. 

The analytic regularity result pertains to 
Leray-Hopf solutions of the viscous, incompressible Navier-Stokes equations 
\eqref{eq:NSEIntr} with sufficiently small data.
We also remark that the variational bootstrapping argument used to establish \eqref{eq:an-P0} 
can also be used to establish analytic regularity for the solution of the 
Navier-Stokes equation in Sobolev spaces without weights, in compact subsets of $P$, 
and up to the boundary in domains with analytic boundary, thereby comprising a ``real-valued''
proof of the analytic regularity results in \cite{Masuda1967a,Giga1983} which were obtained
in these references with function-theoretic arguments.
The proof in the present paper proceeds, instead, by the classical arguments \cite{MorrNiren57}
which are based on localization, a-priori estimates and bootstrapping arguments combined
with induction to control growth of constants in terms of derivative orders. 

The present proof is developed for so-called ``no-slip'' boundary conditions 
(i.e. homogeneous Dirichlet boundary conditions) 
for the velocity field which entailed weighted norms with homogeneous corner weights.
They can be extended verbatim to certain other combinations of boundary conditions at the corners. 
Corresponding weighted shift theorems in \emph{finite order} weighted spaces 
were obtained in \cite{OrltSndig95}. 

The presently obtained analytic regularity shifts for the (Navier-)Stokes equations will
allow for \emph{exponential convergence rates} of suitable high-order and
  model order reduction 
discretizations of these equations. 
See, e.g., \cite{CckbGKanDoMiScho05,CSMSuri99,Kazeev2018,NSVEM18}.
The detailed \emph{a priori} analysis and exponential convergence of a mixed $hp$ discontinuous
  Galerkin discretization of the Navier-Stokes equations in polygons
  \eqref{eq:NSEIntr} are developed in \cite{Schotzau2020}.

\bibliographystyle{siamplain}

\end{document}